\documentclass[12pt]{amsart}
\usepackage{amsmath}	
\usepackage{amssymb}
\usepackage{caption}
\usepackage[labelformat=simple,labelfont={}]{subcaption}

\usepackage{accents}
\usepackage{tikz}
\usetikzlibrary{calc,intersections,arrows,automata}

\newtheorem{theorem}{Theorem}[section]

\newtheorem{lemma}[theorem]{Lemma}

\newtheorem{corollary}[theorem]{Corollary}

\theoremstyle{remark}
\newtheorem*{remark}{Remark}

\theoremstyle{definition}
\newtheorem{definition}[theorem]{Definition}

\numberwithin{equation}{section}

\newcommand{\et}{\quad\mbox{and}\quad}

\newcommand{\bN}{\mathbb{N}}

\newcommand{\bQ}{\mathbb{Q}}
\newcommand{\bR}{\mathbb{R}}
\newcommand{\bZ}{\mathbb{Z}}
\newcommand{\cC}{{\mathcal{C}}}

\newcommand{\cF}{{\mathcal{F}}}
\newcommand{\cK}{{\mathcal{K}}}

\newcommand{\cS}{{\mathcal{S}}}

\newcommand{\Deltabar}{\bar{\Delta}}
\newcommand{\disp}{\displaystyle}

\newcommand{\image}{\mathrm{Im}}

\newcommand{\norm}[1]{\|\hspace*{2pt}#1\hspace*{1pt}\|}

\newcommand{\oA}{\bar{A}}
\newcommand{\oB}{\bar{B}}

\newcommand{\omegahat}{\hat{\omega}}

\newcommand{\kbot}{{\underline{k}}}
\newcommand{\ktop}{{\overline{k}}}
\newcommand{\ellbot}{{\underline{\ell}}}
\newcommand{\elltop}{{\overline{\ell}}}

\newcommand{\tcK}{{\tilde{\cK}}}

\newcommand{\ue}{\mathbf{e}}

\newcommand{\uf}{\mathbf{f}}

\newcommand{\uL}{\mathbf{L}}

\newcommand{\uP}{{\mathbf{P}}}

\newcommand{\uR}{{\mathbf{R}}}
\newcommand{\uS}{{\mathbf{S}}}

\newcommand{\uu}{\mathbf{u}}

\newcommand{\ux}{\mathbf{x}}

\newcommand{\uy}{\mathbf{y}}

\newcommand{\rien}[1]{}

\begin{document}

\baselineskip=15.1pt
%\Large

\title[On Parametric Geometry of Numbers]
{Counter-examples in Parametric Geometry of Numbers}
\author{Martin Rivard-Cooke}
\address{
   D\'epartement de Math\'ematiques\\
   Universit\'e d'Ottawa\\
   150 Louis Pasteur\\
   Ottawa, Ontario K1N 6N5, Canada}
\email{mriva039@uottawa.com}
\author{Damien ROY}
\address{
   D\'epartement de Math\'ematiques\\
   Universit\'e d'Ottawa\\
   150 Louis Pasteur\\
   Ottawa, Ontario K1N 6N5, Canada}
\email{droy@uottawa.ca}


\begin{abstract}
Thanks to recent advances in parametric geometry of
numbers, we know that the spectrum of any set of $m$
exponents of Diophantine approximation to points
in $\bR^n$ (in a general abstract setting) is a
compact connected subset of $\bR^m$.  Moreover, this
set is semi-algebraic and closed under coordinate-wise
minimum for $n\le 3$.  In this paper, we give examples
showing that for $n\ge 4$ each of the latter properties
may fail.
\end{abstract}

\maketitle

%%%%%%%%%%%%%%%%%%%%%%%%%%%%%%%%%%%%%%%%%%%%%%%%%%%%
%
%  Introduction
%
%%%%%%%%%%%%%%%%%%%%%%%%%%%%%%%%%%%%%%%%%%%%%%%%%%%%

\section{Introduction}
\label{sec:intro}

The basic object of Diophantine approximation is
rational approximation to points $\uu$ in $\bR^n$.
This is generally measured by elements of the
extended real line $[-\infty,\infty]$ called
exponents of approximation to $\uu$.  The spectrum
of a family of exponents $(\mu_1, \dots, \mu_m)$
is the subset of $[-\infty,\infty]^m$ consisting
of all $m$-tuples $(\mu_1(\uu), \dots, \mu_m(\uu))$
as $\uu$ varies among the points of $\bR^n$ with
linearly independent coordinates over $\bQ$.
In all cases where such a spectrum has been
explicitly determined, its
trace on $\bR^m$ (the set of its finite points)
can be expressed as the set of common solutions
of a finite system of polynomial inequalities
(called transference inequalities).  In particular
this trace is a \emph{semi-algebraic} subset
of $\bR^m$, namely a finite union of such
solution-sets. It is natural to ask if this is
always so.

A general study of spectra is proposed in \cite{R2017}.
It is based on parametric geometry of numbers and
the observation, due to Schmidt and Summerer
\cite{SS2013a}, that the standard exponents of
approximation to a point $\uu\in\bR^n$ can be
computed from the knowledge of the successive
minima of a certain one parameter family of convex
bodies in $\bR^n$.  Using the equivalent formalism
of \cite{R2015}, we choose the family
\[
 \cC_\uu(q)
  := \{\ux\in\bR^n \,;\,
       \norm{\ux}\le 1\ \text{and}\ |\ux\cdot\uu|\le e^{-q}\}
 \quad (q\ge 0),
\]
where $\ux\cdot\uu$ denotes the usual scalar
product of $\ux$ and $\uu$, and where
$\norm{\ux}=|\ux\cdot\ux|^{1/2}$ is the Euclidean
norm of $\ux$.  For each integer $i$ with
$1\le i\le n$ and each $q\ge 0$, we define
$L_i(q)=\log\lambda_i(q)$ where $\lambda_i(q)$
is the $i$-th minimum of $\cC_\uu(q)$ with respect
to $\bZ^n$, that is the smallest real number
$\lambda\ge 1$ such that $\lambda\cC_\uu(q)$
contains at least $i$ linearly independent points
of $\bZ^n$.  Let $\uL_\uu\colon[0,\infty)\to\bR^n$
be the map given by
\[
 \uL_\uu(q) = (L_1(q),\dots,L_n(q)) \quad (q\ge 0).
\]
Then, each standard exponent of approximation
to $\uu$ can be computed as a linear fractional
function of the quantity
\begin{equation}
\label{intro:eq1}
 \mu_T(\uL_\uu):=\liminf_{q\to\infty} \frac{1}{q}T(\uL_\uu(q))
\end{equation}
for some non-zero linear form $T\colon\bR^n\to\bR$.
For example, as it is explained in \cite{R2016}, the
exponents $\omega_{d-1}(\uu)$ and $\omegahat_{d-1}(\uu)$
introduced by Laurent in \cite{La2009b} for each
integer $d$ with $1\le d\le n-1$, which provide measures
of approximation to $\uu$ by subspaces of $\bR^n$ of
dimension $d$ defined over $\bQ$, can be computed as
\[
 \omega_{d-1}(\uu)=\mu_T(\uL_\uu)^{-1}-1
 \et
 \omegahat_{d-1}(\uu)=-\mu_{-T}(\uL_\uu)^{-1}-1
\]
for the linear form $T=\psi_{n-d}$ given by
$\psi_{n-d}(x_1,\dots,x_n)=\sum_{i=1}^{n-d}x_i$.  This observation
is used in \cite{R2017} to attach an abstract spectrum to each
linear map $T=(T_1,\dots,T_m)\colon\bR^n\to\bR^m$.  It is denoted
$\image^*(\mu_T)$ and consists of all $m$-tuples
\begin{equation}
 \label{intro:eq2}
 \mu_T(\uL_\uu):=(\mu_{T_1}(\uL_\uu),\dots,\mu_{T_m}(\uL_\uu))
\end{equation}
where $\uu$ runs through the points of $\bR^n$ with linearly
independent coordinates over $\bQ$.  We refer the reader to
\cite{R2017} for a short description of the known spectra prior
to 2018. To this list, we should now add the recent breakthrough
of Marnat and Moschevitin \cite{MM2019} who determined the spectra
of the pairs $(\omega_0,\omegahat_0)$ and
$(\omega_{n-2},\omegahat_{n-2})$ by a combination of classical
arguments and of parametric geometry of numbers, thereby proving
a conjecture proposed by Schmidt and Summerer
in \cite[\S 3]{SS2013b}. We also refer
to \cite[Chapter 2]{Ri2019} for a short alternative proof of
this result based only on parametric geometry of numbers together
with a general conjecture about the spectra of the pairs
$(\omega_d,\omegahat_d)$ with $0\le d\le n-2$ and a proof of that
conjecture for $n=4$.

In \cite{R2017} it is shown that, for each linear map
$T\colon\bR^n\to\bR^m$, the spectrum $\cS=\image^*(\mu_T)$ is a
compact connected subset of $\bR^m$ and that, when $n\le 3$,
it is semi-algebraic and closed under coordinate-wise minimum.
The last property means that for any two points
$\ux=(x_1,\dots,x_m)$ and $\uy=(y_1,\dots,y_m)$ in
$\cS$, the point
\begin{equation}
 \label{intro:eq:min}
 \min\{\ux,\uy\}=(\min\{x_1,y_1\},\dots,\min\{x_m,y_m\})
\end{equation}
also belongs to $\cS$.  In this paper we show that both of
these properties fail for $n\ge 4$.  Our counter-examples
involve linear maps $T\colon\bR^n\to\bR^m$ with $m=n+1$.
It would be interesting to know, for given $n\ge 4$, what
is the smallest value of $m$ for which there exists a
linear map $T\colon\bR^n\to\bR^m$ whose corresponding
spectrum is not a semi-algebraic subset of $\bR^m$ or
is not closed under coordinate-wise minimum.
In particular, we wonder if such counter-examples exist
with $m=2$ and, more precisely, if one could take
$T=(F,-F)\colon\bR^n\to\bR^2$ for some linear form $F$ on
$\bR^n$.

%%%%%%%%%%%%%%%%%%%%%%%%%%%%%%%%%%%%%%%%%%%%%%%%%%%%
%
%  Parametric geometry of numbers
%
%%%%%%%%%%%%%%%%%%%%%%%%%%%%%%%%%%%%%%%%%%%%%%%%%%%%

\section{Parametric geometry of numbers}
\label{sec:pgn}

Fix an integer $n\ge 2$.  The main theorem of parametric
geometry of numbers \cite[Theorem 1.3]{R2015} asserts
that, modulo bounded functions, the classes of maps $\uL_\uu$
attached to points $\uu$ in $\bR^n$ are the same as the
classes of rigid $n$-systems of mesh $\delta$ for any given
$\delta\ge 0$.  There are several equivalent ways of defining
an $n$-system (also called $(n,0)$-systems).  One of them
is \cite[Definition 2.8]{R2015} (with $\gamma=0$).  Here we
choose the simpler Definition 2.1 of \cite{R2017} where
\[
 \ue_1=(1,0,\dots,0),\dots,\ue_n=(0,\dots,0,1)
\]
denote the elements of the canonical basis of $\bR^n$.

\begin{definition}
\label{pgn:def:n-system}
Let $I$ be a closed subinterval of $[0,\infty)$ with
non-empty interior.  An $n$-system on $I$ is a map
$\uP=(P_1,\dots,P_n)\colon I \to \bR^n$ with the
property that, for any $q\in I$:
\begin{itemize}
\item[(S1)] $0\le P_1(q)\le\cdots\le P_n(q)$ and
  $P_1(q)+\cdots+P_n(q)=q$;
\smallskip
\item[(S2)] there exist $\epsilon>0$ and integers
  $k,\ell\in\{1,\dots,n\}$ such that
  \[
   \uP(t)=\begin{cases}
         \uP(q)+(t-q)\ue_\ell
          &\text{for any $t\in I\cap[q-\epsilon,q]$,}\\
         \uP(q)+(t-q)\ue_k
          &\text{for any $t\in I\cap[q,q+\epsilon]$;}
         \end{cases}
   \]
\item[(S3)] if $q$ is in the interior of $I$ and if the
  integers $k$ and $\ell$ from (S2) satisfy $k>\ell$,
  then $P_\ell(q)=\cdots=P_k(q)$.
\end{itemize}
We say that $\uP$ is \emph{proper} if $I=[q_0,\infty)$ for some
$q_0\ge 0$ and if $\lim_{q\to\infty}P_1(q)=\infty$.
\end{definition}

As suggested by Luca Ghidelli, one can view an $n$-system
on $[0,\infty)$ as describing a ball game with $n$ players
$P_1,\dots,P_n$ moving on the real line as a function of
the time according to the following rules.
\begin{itemize}
 \item[(R1)] At time $t=0$, all players are at position $0$.
 \item[(R2)] No player can pass another one so that, at any time
   $t\ge 0$, their order remains $P_1\le\cdots\le P_n$.
 \item[(R3)] The only player that can move is the one who
   carries the ball and that player moves with constant speed $1$.
 \item[(R4)] A player can only pass the ball to a player
   that is behind or next to him/her.
\end{itemize}
Indeed, for $I=[0,\infty)$, the rules (R1) to (R3) codify
(S1) and (S2) while (R4) codifies (S3), assuming that the ball
moves instantaneously.  This interpretation
is useful in many ways.  For example, when $n\ge 3$, we obtain
an $(n-1)$-system out of an $n$-system by considering
only the positions of $P_1,\dots,P_{n-1}$ and by stopping
the time counter when $P_n$ has the ball.  Another way
is to consider only the positions of $P_2,\dots,P_{n}$ and
to stop counting the time when $P_1$ has the ball.

Let $\uP=(P_1,\dots,P_n)$ be an $n$-system on an interval
$I$ as in Definition \ref{pgn:def:n-system}.  Following the
terminology of Schmidt and Summerer in \cite{SS2013a}, the
\emph{division numbers} of $\uP$ are the boundary points of $I$
and the interior points $q$ of $I$ at which $\uP$ is not
differentiable, namely those for which we have $k\neq \ell$
in (S2).  The \emph{switch numbers} of $\uP$ are the boundary
points of $I$ and the interior points $q$ of $I$ for which
we have $k<\ell$ in (S2).  The \emph{division points} of $\uP$
(resp.\ the \emph{switch points} of $\uP$) are the values of $\uP$ at
its division numbers (resp.\ switch numbers).  When $I=[0,\infty)$,
the non-zero division points of $\uP$ represent the positions
of the players when the ball is passed from a player to another
one, and the non-zero switch points of $\uP$ represent their
positions when the ball is passed from a player to another
one behind.

\begin{definition}
Let $\delta>0$ and let $q_0\ge 0$.  We say that an $n$-system
$\uP$ on $[q_0,\infty)$ is \emph{rigid of mesh $\delta$} if each
non-zero switch point of $\uP$ has $n$ distinct coordinates
and if these coordinates are integer multiples of $\delta$.
\end{definition}

Equivalently, an $n$-system $\uP\colon[q_0,\infty)\to\bR^n$ is
rigid of mesh $\delta$ if $q_0\in\delta\bZ$, if $\uP(q)\in\delta\bZ^n$
for each $q\in\delta\bZ$ with $q\ge q_0$, and if for $q=q_0$
and each $q\in(q_0,\infty)\setminus\delta\bZ$ the point $\uP(q)$ has
$n$ distinct coordinates.  In particular, the division numbers of
such a system belong to $\delta\bZ$.

The present paper relies on the following consequence of
Theorems 8.1 and 8.2 of \cite{R2015} to which we alluded at
the beginning of the section.

\begin{theorem}
\label{pgn:thm:approxLbyP}
For each non-zero point $\uu$ in $\bR^n$ and each $\delta>0$, there
exist $q_0\in\delta\bZ$ with $q_0\ge 0$ and a rigid $n$-system $\uP$
of mesh $\delta$ on $[q_0,\infty)$ such that $\uP-\uL_\uu$
is bounded on $[q_0,\infty)$.  Conversely, for any $q_0\ge 0$
and any $n$-system $\uP$ on $[q_0,\infty)$, there exists a non-zero
point $\uu\in\bR^n$ such that $\uP-\uL_\uu$ is bounded on
$[q_0,\infty)$.  The point $\uu$ has $\bQ$-linearly independent
coordinates if and only if the map $\uP$ is proper.
\end{theorem}

The last assertion follows from the preceding ones
based on the fact that a point $\uu$ in $\bR^n$ has
$\bQ$-linearly independent coordinates if and only if
the first component of the map $\uL_\uu$ is unbounded
(i.e.\ if $\lim_{q\to\infty}\lambda_1(\cC_\uu(q))=\infty$).

It is interesting to compare the above notion of an
$n$-system to that of an
$1\times(n-1)$-template according to Das, Fishman, Simmons
and Urba\'nski in \cite[Definition 2.1]{DFSU}.  Adapted to
our present context, it becomes exactly the notion of a
generalized $n$-system as in \cite[Definition 4.5]{R2016}.
The formulation given below follows the clever and
concise definition of a template by the four authors.

\begin{definition}
\label{pgn:def:templates}
Let $I$ be a closed subinterval of $[0,\infty)$ with
non-empty interior.  A generalized $n$-system on
$I$ is a continuous piecewise linear map
$\uP=(P_1,\dots,P_n)\colon I \to \bR^n$
with the following properties.
\begin{itemize}
\item[(G1)] For each $q\in I$, we have
  $0\le P_1(q)\le\cdots\le P_n(q)$
  and
  $P_1(q)+\cdots+P_n(q)=q$.
\item[(G2)] For each $j=1,\dots,n$, the component
  $P_j\colon I\to\bR$ is both monotone increasing and $1$-Lipschitz.
\item[(G3)] For each $j=1,\dots,n-1$ and each subinterval
  $H$ of $I$ on which $P_j<P_{j+1}$, the
  sum $P_1+\cdots+P_j$ is convex on $H$ with
  slopes in $\{0,1\}$.
\end{itemize}
We say that $\uP$ is \emph{proper} if $I=[q_0,\infty)$ for some
$q_0\ge 0$ and if $\lim_{q\to\infty}P_1(q)=\infty$.
\end{definition}

Recall that a function $f\colon I\to\bR$ is $1$-Lipschitz
if it satisfies $f(b)-f(a)\le b-a$ for any $a,b\in I$ with
$a\le b$.  So (G2) amounts to asking that each $P_j$ has
slopes belonging to $[0,1]$.

To analyze this definition and compare it to
\cite[Definition 4.5]{R2016}, fix such a map $\uP$.
Set $M_0=0$ and $M_j=P_1+\cdots+P_j$ for each
$j=1,\dots,n$.  Then, consider a non-empty open
subinterval $H$ of $I$ on which $\uP$ is
affine.  For each $j=1,\dots,n-1$, we have either
$P_j=P_{j+1}$ or $P_j<P_{j+1}$ on the whole of $H$.
In the latter case, $M_j$ has constant slope $0$ or
$1$ on $H$ by (G3).  Let $\kbot\ge 1$ be the largest
index for which $M_{\kbot-1}$ has slope $0$ on $H$,
and let $\ktop\le n$ be the smallest one for which
$M_{\ktop}$ has slope $1$ on $H$.  Then, for each
index $j$ with $\kbot\le j<\ktop$, the function
$M_j$ has constant slope $M_j'\in(0,1)$ (because of (G2)),
and so $P_j=P_{j+1}$ on $H$.  Thus
$P_{\kbot},\dots,P_{\ktop}$ coincide and have slope
$1/(\ktop-\kbot+1)$ on $H$ while all other components
of $\uP$ are constant on $H$.

Now, consider an
interior point $q$ of $I$ at which $\uP$ is not
differentiable and choose $\epsilon>0$ such that
$\uP$ is defined and differentiable on both
$(q-\epsilon,q)$ and $(q,q+\epsilon)$.  For each
$j=1,\dots,n-1$ such that $P_j(q)<P_{j+1}(q)$, we
have $P_j<P_{j+1}$ on $(q-\epsilon,q+\epsilon)$
and so $M_j$ is convex with slopes in $\{0,1\}$
on that interval: either it has constant slope
on $(q-\epsilon,q+\epsilon)$ or else it has slope
$0$ on $(q-\epsilon,q)$ and slope $1$ on
$(q,q+\epsilon)$.  Let
$\ellbot\le\elltop$ and $\kbot\le\ktop$ be the
indices for which $P_{\ellbot}=\cdots=P_{\elltop}$
have slope $1/(\elltop-\ellbot+1)$ on $(q-\epsilon,q)$,
and $P_{\kbot}=\cdots=P_{\ktop}$
have slope $1/(\ktop-\kbot+1)$ on $(q,q+\epsilon)$.
Then we have
\[
 P_{\ellbot}(q)=\cdots=P_{\ktop}(q)
 \quad
 \text{if}
 \quad
 \ellbot<\ktop
\]
because for each $j$ with $\ellbot\le j<\ktop$,
the function $M_j$ has slope $>0$ on $(q-\epsilon,q)$
and slope $<1$ on $(q,q+\epsilon)$, and so
$P_j(q)=P_{j+1}(q)$ by a previous observation.

A generalized $n$-system on $[0,\infty)$ can
therefore be viewed as describing a ball game where
several players may carry the ball together
(like young children generally do).
We keep the same rules (R1) and (R2) but replace (R3)
and (R4) by the following weaker rules.
\begin{itemize}
 \item[(R3')] Only the players that carry the ball can move,
   and they move together at speed $1/m$ where $m$ is the
   size of their group.
 \item[(R4')] The group of players carrying the ball can only
   pass the ball to a group of players that are waiting behind
   them or are next to them.
\end{itemize}
It follows from this interpretation that each $n$-system is a
generalized $n$-system and that any generalized
$n$-system is a uniform limit of $n$-systems (see
\cite[Corollary 4.7]{R2016}).  Thus Theorem \ref{pgn:thm:approxLbyP}
admits the following complement.

\begin{theorem}
\label{pgn:thm:approxGen}
For any $q_0\ge 0$ and any generalized $n$-system $\uP$
on $[q_0,\infty)$, there exists a non-zero
point $\uu\in\bR^n$ such that $\uP-\uL_\uu$ is bounded on
$[q_0,\infty)$.  The point $\uu$ has $\bQ$-linearly independent
coordinates if and only if the map $\uP$ is proper.
\end{theorem}

The fact that an $n$-system has property (G3) is very useful
and we will often use it in Sections \ref{sec:min} and
\ref{sec:non-semi-alg}.
In terms of a team of players following the rules (R1)--(R4),
it simply expresses the fact that, for a given integer $j$
with $1\le j<n$, when one of $P_1,\dots,P_j$ gets the ball,
the ball remains within that group until $P_j$ meets $P_{j+1}$
with the ball.

%%%%%%%%%%%%%%%%%%%%%%%%%%%%%%%%%%%%%%%%%%%%%%%%%%%%
%
%  Computing spectra from $n$-systems
%
%%%%%%%%%%%%%%%%%%%%%%%%%%%%%%%%%%%%%%%%%%%%%%%%%%%%

\section{Computing spectra from $n$-systems}
\label{sec:CompSpec}

Let $n\ge 2$ be an integer.  For any $q_0\ge 0$ and any
Lipschitz map
$\uP\colon [q_0,\infty)\to\bR^n$, we denote by
$\cF(\uP)$ the set of accumulation points of the quotients
$q^{-1}\uP(q)$ as $q$ goes to infinity, and define $\cK(\uP)$
to be the convex hull of $\cF(\uP)$, as in \cite[\S 3]{R2017}.
When $\uP$ is an $n$-system or a generalized $n$-system, the set
$\cF(\uP)$ is contained in
\[
 \Deltabar
   :=\{(x_1,\dots,x_n)\in\bR^n\,;\,
         0\le x_1\le\cdots\le x_n
         \text{ and }
         x_1+\cdots+x_n=1\}.
\]
As this is a compact convex subset of $\bR^n$, both sets $\cF(\uP)$
and $\cK(\uP)$ are then compact subsets of $\Deltabar$.

For each integer $m\ge 1$, we equip $\bR^m$ with the
coordinate-wise ordering where
\[
 (x_1,\dots,x_m) \le (y_1,\dots,y_m)
 \quad\Longleftrightarrow\quad
 x_1\le y_1,\,\dots,\,x_m\le y_m.
\]
For that partial order, the minimum of two points is their coordinate-wise
minimum as defined in the introduction, and every bounded
subset of $\bR^m$ has an infimum in $\bR^m$.  Then, for any linear
map $T=(T_1,\dots,T_m)\colon\bR^n\to\bR^m$ and any Lipschitz map
$\uP\colon[q_0,\infty)\to\bR^n$, we define
\begin{align*}
 \mu_T(\uP)
   = \inf T(\cK(P))
   &= \inf T(\cF(\uP))\\
   &=\liminf_{q\to\infty} q^{-1}T(\uP(q))\\
   &=\big( \liminf_{q\to\infty} q^{-1}T_1(\uP(q)),
          \dots,
          \liminf_{q\to\infty} q^{-1}T_m(\uP(q)) \big)
\end{align*}
as in \cite[\S3]{R2017}.  In view of Theorems \ref{pgn:thm:approxLbyP}
and \ref{pgn:thm:approxGen}, the computation of a spectrum is reduced
to a problem about maps of a combinatorial nature.

\begin{theorem}
\label{comput:thm}
Let $\delta>0$.  The spectrum $\image^*(\mu_T)$ of a linear
map $T\colon\bR^n\to\bR^m$ is the set of all numbers $\mu_T(\uP)$
where $\uP\colon[q_0,\infty)\to\bR^n$ is a proper rigid $n$-system
of mesh $\delta$ (resp.~a proper $n$-system, resp.~a proper
generalized $n$-system).
\end{theorem}

For the purpose of this paper, we will need the following facts.

\begin{lemma}
\label{comput:lemma}
Let $q_0\ge 0$ and let $\uP\colon[q_0,\infty)\to\bR^n$ be a proper
generalized $n$-system.
\begin{itemize}
\item[(i)] Let $w_1<w_2<\cdots$ denote the points of $(q_0,\infty)$ at
which $\uP$ is not differentiable, listed in increasing order, and let
$E$ be the set of limit points of the sequence
$(w_i^{-1}\uP(w_i))_{i\ge 1}$.  Then $\cK(\uP)$ is the convex hull
of $E$.
\item[(ii)] For each $\delta>0$ there exists $Q_\delta\in (q_0,\infty)$
such that for each $q\ge Q_\delta$ we have
  \[
    q^{-1}\uP(q)
     \in\cF(\uP)+[-\delta,\delta]^n
        :=\{\ux\in\bR^n\,;\, \norm{\ux-\uy}\le \delta
            \text{ for some } \uy\in\cF(\uP)\}.
  \]
\item[(iii)] If there exists $\rho>1$ such that $\uP(\rho q)=\rho\uP(q)$
  for each $q\ge q_0$, then we have
  \[
    \cF(\uP)=\{q^{-1}\uP(q)\,;\,q_0\le q\le\rho q_0\}.
  \]
\end{itemize}
\end{lemma}

The property (i) is proved for proper $n$-systems in
\cite[Proposition 3.2]{R2017} but the proof extends to
generalized $n$-systems as it relies simply on the fact that,
for each $i\ge 1$, the restriction of $\uP$ to $[w_i,w_{i+1}]$ is
an affine map and thus $\{t^{-1}\uP(t)\,;\,w_i\le t\le w_{i+1}\}$
is the line segment joining $w_i^{-1}\uP(w_i)$ to
$w_{i+1}^{-1}\uP(w_{i+1})$ in $\Deltabar$.  Similarly, (ii) is
proved for $n$-systems in \cite[Lemma 4.1]{R2017}
but the proof, based on a compactness argument, applies in fact to
any Lipschitz map.  Finally (iii) is clear from the definition.
A generalized $n$-system $\uP$ which satisfies the condition
in (iii) for some $\rho>1$ is called \emph{self-similar}.

%%%%%%%%%%%%%%%%%%%%%%%%%%%%%%%%%%%%%%%%%%%%%%%%%%%%
%
%  Examples of spectra which are not closed under the minimum
%
%%%%%%%%%%%%%%%%%%%%%%%%%%%%%%%%%%%%%%%%%%%%%%%%%%%%

\section{Examples of spectra which are not closed under the minimum}
\label{sec:min}

For simplicity, we only give an example in dimension $n=4$.
We will construct proper generalized $4$-systems $\uR$
and $\uS$, and a linear map $T\colon\bR^4\to\bR^5$ such that
$\min\{\mu_T(\uR),\mu_T(\uS)\}$ is not in the spectrum of $T$.

Note that, in dimension $4$, the set
\[
 \Deltabar=\{(x_1,\dots,x_4)\in\bR^4\,;\,
   0\le x_1\le\cdots\le x_4
   \text{ and }
   x_1+\cdots+x_4=1\}
\]
is a tetrahedron with vertices
\[
 E_1=\Big(\frac{1}{4},\frac{1}{4},\frac{1}{4},\frac{1}{4}\Big),\
 E_2=\Big(0,\frac{1}{3},\frac{1}{3},\frac{1}{3}\Big),\
 E_3=\Big(0,0,\frac{1}{2},\frac{1}{2}\Big)
 \text{ and }
 E_4=(0,0,0,1).
\]
For each $i=1,2,3$, we denote by $\Deltabar_i$ the face of
$\Deltabar$ consisting of the points $(x_1,\dots,x_4)$ in
$\Deltabar$ with $x_i=x_{i+1}$.  The fourth face of $\Deltabar$
is the triangle $\Deltabar_0=E_2E_3E_4$ defined
by $x_1=0$.

Let $\alpha,\beta\in\bR$ with $1<\alpha<\beta$.  We first
observe that there is a unique generalized $4$-system
$\uR=(R_1,\dots,R_4)$ on $[3+\alpha,\alpha(3+\alpha)]$
with $R_1=R_2$, whose division points are
\[
 A_1=(1,1,1,\alpha),\
 A_2=(1,1,\alpha,\alpha),\
 A_3=(1,1,\alpha,\alpha^2),\
 \alpha A_1=(\alpha,\alpha,\alpha,\alpha^2).
\]
Its combined graph is shown on the left in Figure \ref{figRS}.
Moreover, this map extends uniquely to a self-similar generalized
$4$-system on $[3+\alpha,\infty)$ also denoted $\uR$ such that
$\uR(\alpha q)=\alpha\uR(q)$ for each $q\ge 3+\alpha$.

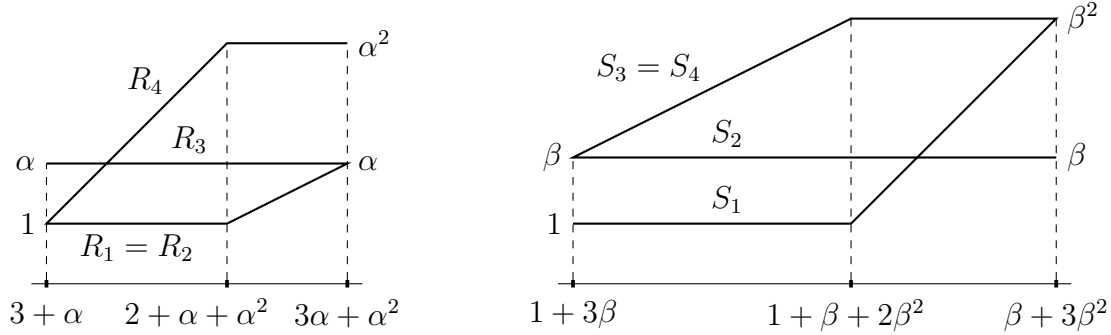
\begin{figure}[h]
   \begin{tikzpicture}[xscale=1,yscale=1]
      %
      % graphe de R
      \pgfmathparse{2}\let\a\pgfmathresult;
      \pgfmathparse{0.8}\let\r\pgfmathresult;
      \pgfmathparse{\r*\a^2-\r}\let\u\pgfmathresult;
      \pgfmathparse{\u+2*(\r*\a-\r)}\let\v\pgfmathresult;
      \draw[thick] (0,\r*\a) -- (\v,\r*\a) -- (\u,\r) --
         (0,\r) -- (\u,\r*\a^2) -- (\v,\r*\a^2);
       \node[below] at (\u/2,\r) {$R_1=R_2$};
       \node[above] at (\u/2+0.7,\r*\a) {$R_3$};
       \node[left] at (\u/2+0.5,\r*\a+1.1) {$R_4$};
       \node[left] at (0,\r) {$1$};
       \node[left] at (0,\r*\a) {$\alpha$};
       \node[right] at (\v,\r*\a) {$\alpha$};
       \node[right] at (\v,\r*\a^2) {$\alpha^2$};
       \draw[line width=0.05cm] (0,0.05) -- (0,-0.05);
       \node[below] at (0,-0.1) {$3+\alpha$};
       \draw[dashed] (0,0.05) -- (0,\r*\a);
       \draw[line width=0.05cm] (\u,0.05) -- (\u,-0.05);
       \node[below] at (\u-0.4,-0.05){$2+\alpha+\alpha^2$};
       \draw[dashed] (\u,0.05) -- (\u,\r*\a^2);
       \draw[line width=0.05cm] (\v,0.05) -- (\v,-0.05)
         node[below] {$3\alpha+\alpha^2$};
       \draw[dashed] (\v,0.05) -- (\v,\r*\a^2);
      \draw (-0.2,0) -- (\v+0.2,0); % node[below]{$q$};
      %
      % graphe de S
      \pgfmathparse{2.1}\let\b\pgfmathresult;
      \pgfmathparse{\v+3}\let\s\pgfmathresult;
      \pgfmathparse{\s+2*(\r*\b^2-\r*\b)}\let\u\pgfmathresult;
      \pgfmathparse{\u+(\r*\b^2-\r)}\let\v\pgfmathresult;
      \draw[thick] (\s,\r) -- (\u,\r) -- (\v,\r*\b^2)
         -- (\u,\r*\b^2) -- (\s,\r*\b) -- (\v,\r*\b);
       \node[left] at (\s/2+\u/2,\r*\b+1.2) {$S_3=S_4$};
       \node[above] at (\s/2+\u/2+0.2,\r*\b) {$S_2$};
       \node[above] at (\s/2+\u/2+0.2,\r) {$S_1$};
       \node[left] at (\s,\r) {$1$};
       \node[left] at (\s,\r*\b) {$\beta$};
       \node[right] at (\v,\r*\b) {$\beta$};
       \node[right] at (\v,\r*\b^2) {$\beta^2$};
       \draw[line width=0.05cm] (\s,0.05) -- (\s,-0.05)
         node[below] {$1+3\beta$};
       \draw[dashed] (\s,0.05) -- (\s,\r*\b);
       \draw[line width=0.05cm] (\u,0.05) -- (\u,-0.05)
         node[below] {$1+\beta+2\beta^2\ $};
       \draw[dashed] (\u,0.05) -- (\u,\r*\b^2);
       \draw[line width=0.05cm] (\v,0.05) -- (\v,-0.05)
         node[below] {$\beta+3\beta^2$};
       \draw[dashed] (\v,0.05) -- (\v,\r*\b^2);
      \draw (\s-0.2,0) -- (\v+0.2,0); % node[below]{$q$};
  \end{tikzpicture}
\caption{The graphs of $\uR$ and $\uS$.}
\label{figRS}
\end{figure}

Similarly there is a unique generalized $4$-system
$\uS=(S_1,\dots,S_4)$ on $[1+3\beta,\beta(1+3\beta)]$
with $S_3=S_4$, whose division points are
\[
 B_1=(1,\beta,\beta,\beta),\
 B_2=(1,\beta,\beta^2,\beta^2),\
 B_3=(\beta,\beta,\beta^2,\beta^2),\
 \beta B_1=(\beta,\beta^2,\beta^2,\beta^2).
\]
Its combined graph is shown on the right in Figure \ref{figRS}
and it extends uniquely to a self-similar generalized
$4$-system on $[1+3\beta,\infty)$ such that
$\uS(\beta q)=\beta\uS(q)$ for each $q\ge 1+3\beta$.

For each $i=1,2,3$, let $\oA_i$ (resp.~$\oB_i$) denote the
quotient of $A_i$ (resp.~$B_i$) by the sum $|A_i|$ (resp.~
$|B_i|$) of its
coordinates.  Since $\uR$ is self-similar with $R_1=R_2$,
it follows from Lemma \ref{comput:lemma} (i) that
$\cK(\uR)$ is the triangle $\oA_1\oA_2\oA_3$ contained
in the face $\Deltabar_1=E_1E_3E_4$ of $\Deltabar$.
Similarly since $S_3=S_4$, the convex set $\cK(\uS)$
is the triangle $\oB_1\oB_2\oB_3$ contained in
$\Deltabar_3=E_1E_2E_3$. These two triangles are shown
on the left drawing in Figure \ref{figK}.

\begin{figure}[h]
\centering
   \begin{tikzpicture}[scale=2]
      %
      % graphe de R
      \def\s{-3}; % abscisse de e1
      \def\r{7}; % distance de e1 a e2
      \def\rbis{2.5}; % distance de e1 a e4 %3.5
      \def\angle{-37.5}; % angle e1 -- e2
      \def\a{2};
      \def\b{4};
      \pgfmathparse{60+\angle-10}\let\anglebis\pgfmathresult; %+5
      \coordinate (ei) at (\s,0);
      \coordinate (eii) at ($(ei)+\r*({cos(\angle)},{sin(\angle)})$);
      \coordinate (eiii) at ($(ei)+\r*({cos(\angle+60)},{sin(\angle+60)})$);
      \coordinate (eiv) at ($(ei)+\rbis*({cos(\anglebis)},{sin(\anglebis)})$);
      \coordinate (Ei) at ($1/4*(ei)+1/4*(eii)+1/4*(eiii)+1/4*(eiv)$);
      \node[below] at (Ei) {$E_1$};
      \coordinate (Eii) at ($1/3*(eii)+1/3*(eiii)+1/3*(eiv)$);
      \node[below] at (Eii) {$E_2$};
      \coordinate (Eiii) at ($1/2*(eiii)+1/2*(eiv)$);
      \node[above] at (Eiii) {$E_3$};
      \coordinate (Eiv) at (eiv);
      \node[left] at (Eiv) {$E_4$};
      \coordinate (A) at ($(ei)+(eii)+(eiii)+\a*(eiv)$);
      \coordinate (Ai) at ($1/(3+\a)*(A)$);
      \node[left] at ($(Ai)+(0,-0.05)$) {$A_1$};
      \coordinate (A) at ($(ei)+(eii)+\a*(eiii)+\a*(eiv)$);
      \coordinate (Aii) at ($1/(2+2*\a)*(A)$);
      \node[right] at ($(Aii)+(-0.05,-0.1)$) {$A_2$};
      \coordinate (A) at ($(ei)+(eii)+\a*(eiii)+\a^2*(eiv)$);
      \coordinate (Aiii) at ($1/(2+\a+\a^2)*(A)$);
      \node[above] at (Aiii) {$A_3$};
      \coordinate (B) at ($(ei)+\b*(eii)+\b*(eiii)+\b*(eiv)$);
      \coordinate (Bi) at ($1/(1+3*\b)*(B)$);
      \node[below] at (Bi) {$B_1$};
      \coordinate (B) at ($(ei)+\b*(eii)+\b^2*(eiii)+\b^2*(eiv)$);
      \coordinate (Bii) at ($1/(1+\b+2*\b^2)*(B)$);
      \node[above] at ($(Bii)+(-0.05,0)$) {$B_2$};
      \coordinate (B) at ($\b*(ei)+\b*(eii)+\b^2*(eiii)+\b^2*(eiv)$);
      \coordinate (Biii) at ($1/(2*\b+2*\b^2)*(B)$);
      \node[left] at (Biii) {$B_3$};
      \draw[thick] (Ei) -- (Eii) -- (Eiii) -- (Eiv) --(Ei) -- (Eiii);
      \draw[thick] (Ai) -- (Aii) -- (Aiii) -- (Ai);
      \path [thick,fill=gray!30,opacity=0.8] (Ai) -- (Aii) -- (Aiii) -- (Ai);
      \draw[thick] (Bi) -- (Bii) -- (Biii) -- (Bi);
      \path [thick,fill=gray!30,opacity=0.8] (Bi) -- (Bii) -- (Biii) -- (Bi);
      \coordinate (m) at ($1/3*(Ei)+2/3*(Eiv)$);
        \node[left] at ($(m)-(0.2,0.1)$) {$\Deltabar_1$};
        \draw[->] ($(m)-(0.2,0.1)$) to [out=0,in=-90, looseness=1] ($(m)+(0.1,0.1)$);
      \coordinate (m) at ($3/4*(Eii)+1/4*(Eiii)$);
        \node[right] at ($(m)+(0.2,0.1)$) {$\Deltabar_3$};
        \draw[->] ($(m)+(0.2,0.1)$) to [out=180,in=45, looseness=1] ($(m)-(0.1,0)$);
     % \draw[->] ($1/2*(Eii)+1/2*(Eiii)$) to [out=-10,in=115, looseness=1] ($\Deltabar_3$);
      % pour le dessin
      % \draw[dashed] (Ei) -- (ei) -- (eii) -- (Eii);
      % \coordinate (F) at ($1/2*(ei)+1/2*(eii)$);
      % \draw[dashed] (Ei) -- (F);
      % \draw[dashed] (eiii) -- (eiv);
      % \draw[dashed] (Ai) -- (eiii);
      % \draw[dashed] (Aii) -- (eiv);
      % \draw[dashed] (Aiii) -- (F);
      % \draw[dashed] (Bi) -- (Eiii);
      % \draw[dashed] (Bii) -- (ei);
      % \draw[dashed] (Biii) -- (eii);
      % \draw[dashed] (ei) -- (10,0); test horizontalite
  \end{tikzpicture}
  \hspace*{2cm}
  \begin{tikzpicture}[scale=2]
      %
      % graphe de R
      \def\s{-3}; % abscisse de e1
      \def\r{7}; % distance de e1 a e2
      \def\rbis{2.5}; % distance de e1 a e4  %3.5
      \def\angle{-37.5}; % angle e1 -- e2
      \def\a{2};
      \def\b{4};
      \pgfmathparse{60+\angle-10}\let\anglebis\pgfmathresult; %+5
      \coordinate (ei) at (\s,0);
      \coordinate (eii) at ($(ei)+\r*({cos(\angle)},{sin(\angle)})$);
      \coordinate (eiii) at ($(ei)+\r*({cos(\angle+60)},{sin(\angle+60)})$);
      \coordinate (eiv) at ($(ei)+\rbis*({cos(\anglebis)},{sin(\anglebis)})$);
      \coordinate (Ei) at ($1/4*(ei)+1/4*(eii)+1/4*(eiii)+1/4*(eiv)$);
      \coordinate (Eii) at ($1/3*(eii)+1/3*(eiii)+1/3*(eiv)$);
      \coordinate (Eiii) at ($1/2*(eiii)+1/2*(eiv)$);
      \node[above] at (Eiii) {$E_3$};
      \coordinate (Eiv) at (eiv);
      \coordinate (A) at ($(ei)+(eii)+(eiii)+\a*(eiv)$);
      \coordinate (Ai) at ($1/(3+\a)*(A)$);
      \filldraw (Aii) circle (0.5pt);
      \node[left] at ($(Ai)+(0,-0.05)$) {$A_1$};
      \coordinate (A) at ($(ei)+(eii)+\a*(eiii)+\a*(eiv)$);
      \coordinate (Aii) at ($1/(2+2*\a)*(A)$);
      \node[right] at (Aii) {$A_2$};
      \coordinate (A) at ($(ei)+(eii)+\a*(eiii)+\a^2*(eiv)$);
      \coordinate (Aiii) at ($1/(2+\a+\a^2)*(A)$);
      \node[left] at (Aiii) {$A_3$};
      \coordinate (B) at ($(ei)+\b*(eii)+\b*(eiii)+\b*(eiv)$);
      \coordinate (Bi) at ($1/(1+3*\b)*(B)$);
      \node[right] at (Bi) {$B_1$};
      \coordinate (B) at ($(ei)+\b*(eii)+\b^2*(eiii)+\b^2*(eiv)$);
      \coordinate (Bii) at ($1/(1+\b+2*\b^2)*(B)$);
      \filldraw (Bii) circle (0.5pt);;
      \node[right] at (Bii) {$B_2$};
      \coordinate (B) at ($\b*(ei)+\b*(eii)+\b^2*(eiii)+\b^2*(eiv)$);
      \coordinate (Biii) at ($1/(2*\b+2*\b^2)*(B)$);
      \node[right] at ($(Biii)+(-0.05,-0.05)$) {$B_3$};
      %\draw[->] ($(Biii)+(-0.2,0.2)$) to [out=-10,in=115, looseness=1] ($(Biii)+(-0.04,0.04)$);
      \filldraw (Biii) circle (0.5pt);
      \draw[dashed] (Ai) -- (Aii) -- (Eiii);
      \draw[dashed] (Bi) -- (Aii);
      \draw[thick] (Bi) -- (Eiii) -- (Aiii) -- (Bi);
      \draw[semithick] (Bi) -- (Ai) -- (Aiii);
     % \draw[dotted, semithick] (Aiii) -- (Biii) -- (Bii) -- (Aiii);
  \end{tikzpicture}
\caption{On the left: the triangles $\cK(R)$ and $\cK(S)$.  On the right:
the convex $\cK$.}
  \label{figK}
\end{figure}
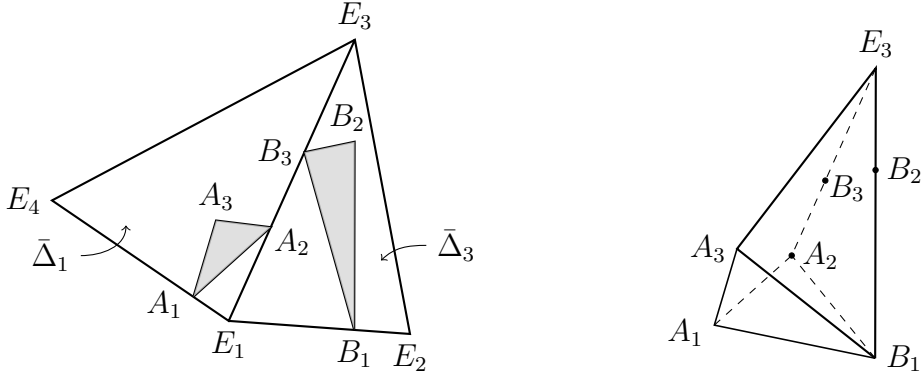

Let $\cK$ denote the convex hull of the set
$S:=\{\oB_1,\oA_1,\oA_2,\oA_3,E_3\}$.  Since $\oB_2$
and $\oB_3$ belong respectively to the line segments
$\oB_1E_3$ and $\oA_2E_3$, the convex $\cK$ contains
both $\cK(\uR)$ and $\cK(\uS)$.  The right drawing in
Figure \ref{figK} shows a picture of $\cK$. Based on the
relative positions of the points of $S$, we see that
$S$ is the set of vertices of $\cK$ and that
the boundary of $\cK$ consists of four triangles
$\oB_1\oA_1\oA_2$, $\oB_1\oA_1\oA_3$, $\oB_1\oA_3E_3$ and
$\oB_1\oA_2E_3\subset\Deltabar_3$, and one quadrilateral
$\oA_1\oA_2E_3\oA_3\subset\Deltabar_1$.

Consider the linear map $T=(T_1,\dots,T_5)\colon\bR^4\to\bR^5$
whose components are given by
\begin{align*}
 T_1(\ux)
   &= -(\alpha-1)\beta x_1 - (\beta-\alpha) x_2 + (\beta-1)x_4,\\
 T_2(\ux)
   &= (\alpha-1) \beta x_1 - (\alpha-1)\beta x_2
      + \alpha(\beta-1)x_3  - (\beta-1)x_4,\\
 T_3(\ux)
   &= \alpha\beta(\alpha-1)x_1 - \alpha(\alpha-1)x_2
      + (\beta-1)x_3 - (\beta-1)x_4,\\
 T_4(\ux) &= x_3-x_4,\\
 T_5(\ux) &= -x_2 + x_4.
\end{align*}
The maps $T_1$, $T_2$ and $T_3$ are chosen so that they
are non-negative on $\cK$ and vanish respectively on the
triangles $\oB_1\oA_1\oA_2$, $\oB_1\oA_1\oA_3$ and
$\oB_1\oA_3E_3$.  As the two other faces of $\cK$ are
contained on the faces $\Deltabar_1$ and $\Deltabar_3$
of $\Deltabar$, we conclude that
\begin{equation}
 \label{min:eq:K}
 \cK=\{\ux\in\Deltabar\,;\, T_i(\ux)\ge 0 \text{ for } i=1,2,3\}.
\end{equation}

We will prove the following result.

\begin{theorem}
\label{min:thm}
Suppose that a proper $4$-system $\uP$ satisfies
$\oB_1\in\cK(\uP)\subseteq\cK$.  Then, we have
$\cK(\uP)\subseteq\Deltabar_3$.
\end{theorem}

In particular, this implies that there is no $4$-system $\uP$
for which $\cK(\uP)$ is the convex hull of $\cK(\uR)\cup\cK(\uS)$.
This requires that the parameters $\alpha$ and $\beta$ satisfy
our current hypothesis $1<\alpha<\beta$ because, for a choice of
parameters satisfying $1<\beta<\alpha$, the first author proved
(unpublished work) that, on the contrary, such a $4$-system $\uP$
exists and so satisfies $\mu_L(\uP)=\min\{\mu_L(\uR),\mu_L(\uS)\}$
for any linear map $L\colon\bR^4\to\bR^m$.

If we take Theorem \ref{min:thm} for granted, we deduce
that the spectrum of $T$ is not closed under the
minimum.

\begin{corollary}
\label{min:cor}
There exists no proper $4$-system $\uP$ such that
$\mu_T(\uP)=\min\{\mu_T(\uR),\mu_T(\uS)\}$.
\end{corollary}

\begin{proof}
We find that
\begin{align*}
 \mu_T(\uR)
  &= \min_{1\le i\le 3} T(\oA_i)
  = \big(0,\,0,\,0,\,\alpha(1-\alpha)|A_3|^{-1},\,
         (\alpha-1)|A_2|^{-1}\big),\\
 \mu_T(\uS)
  &= \min_{1\le i\le 3} T(\oB_i)
  = \big(0,\,0,\,0,\,0,\,0\big),
\end{align*}
and thus $\min\{\mu_T(\uR),\mu_T(\uS)\}=(0,0,0,c,0)$ where
$c=\alpha(1-\alpha)|A_3|^{-1}$ is negative.  Suppose on the
contrary that there exists a $4$-system $\uP$ such that
$\mu_T(\uP)=(0,0,0,c,0)$.  Then we have $\inf T_i(\cK(\uP))=0$
for $i=1,2,3$ and so $\cK(\uP)\subseteq \cK$ because of
\eqref{min:eq:K}.  We also have $\inf T_5(\cK(\uP))=0$.
However, $T_5$ is strictly positive at each vertex of $\cK$
(the points of $S$) except at $\oB_1$ where it vanishes.
Thus $\oB_1$ is the only point of $\cK$ where $T_5$ vanishes
and so $\cK(\uP)$ must contain $\oB_1$.  Finally, we have
$\inf T_4(\cK(\uP))=c<0$ and so $\cK(\uP)\nsubseteq\Deltabar_3$
because $T_4$ vanishes everywhere on $\Deltabar_3$. This
contradicts Theorem \ref{min:thm}.
\end{proof}

Clearly, the corollary requires that there is no $4$-system $\uP$
for which $\cK(\uP)$ is the convex hull $\tcK$ of $\cK(\uR)\cup\cK(\uS)$.
Conversely, if we only assume this fact, then we can construct a
linear map $T\colon\bR^4\to\bR^{10}$ for which $\min\{\mu_T(\uR),\mu_T(\uS)\}$
is not in the spectrum of $T$.  It suffices to choose $T_1,\dots,T_4$
so that the conditions $T_i\ge 0$ ($1\le i\le 4$) define $\tcK$
within $\Deltabar$ and to choose the remaining six components
$T_5,\dots,T_{10}$ so that they vanish at one of the six vertices
$\oA_1,\oA_2,\oA_3,\oB_1,\oB_2,\oB_3$ of $\tcK$ and are strictly
positive at the other vertices.  The construction that we propose
here is more economical as it uses a linear map $T$ with only five
components.

\medskip
We now turn to the proof of Theorem \ref{min:thm}. From now on we
fix $q_0>0$ and a proper $4$-system
$\uP=(P_1,\dots,P_4)\colon[q_0,\infty)\to\bR^4$
satisfying $\oB_1\in\cK(\uP)\subset\cK$, as in the statement of the
theorem. For each $q\ge q_0$, we set
\begin{equation}
 \label{min:eq:kappa}
 \kappa_1(q)=\frac{\beta P_1(q) - P_2(q)}{q},
 \quad
 \kappa_2(q)=\frac{P_4(q) - P_2(q)}{q},
 \quad
 \kappa_3(q)=\frac{P_4(q) - P_3(q)}{q}.
\end{equation}
We note that
\begin{equation}
 \label{min:eq:equiv}
 \cK(\uP)\subseteq\Deltabar_3
 \iff
 \cF(\uP)\subseteq\Deltabar_3
 \iff
 \lim_{q\to\infty}\kappa_3(q)=0.
\end{equation}
So we are left to showing that $\kappa_3$ vanishes at infinity.

We first rewrite the formula $q=P_1(q)+\cdots+P_4(q)$ ($q\ge q_0$)
as follows.

\begin{lemma}
\label{min:lemma:q}
We have
$q=(1+3\beta)P_1(q)+(2\kappa_2(q)-\kappa_3(q)-3\kappa_1(q))q$
for each $q\ge q_0$.
\end{lemma}

For each $\delta>0$, we choose $Q_\delta > q_0$ as in
Lemma \ref{comput:lemma} (ii) so that
$q^{-1}\uP(q)\in\cF(\uP)+[-\delta,\delta]^4$ for each $q\ge Q_\delta$.
We will need the following estimates.

\begin{lemma}
\label{min:lemma:kappa}
There exists a constant $c>0$ with the following property.
For each $\delta>0$ and each $q\ge Q_\delta$, we have
\begin{itemize}
 \item[(i)]
  $\disp -c\delta \le \kappa_1(q)
    \le \frac{\beta-1}{\alpha-1}\kappa_2(q)+c\delta$,
 \item[(ii)]
  $\disp \kappa_3(q)
    \le \frac{\alpha^2-\alpha}{\beta-1}\kappa_1(q)+c\delta$.
\end{itemize}
\end{lemma}

\begin{proof}
Consider the linear forms $f_1,f_2,f_3$ on $\bR^4$ given by
\begin{equation}
\label{min:lemma:kappa:eq1}
 f_1(\ux)=\beta x_1-x_2,
 \quad
 f_2(\ux)=x_4-x_2
 \et
 f_3(\ux)=x_4-x_3.
\end{equation}
A quick computation shows that
\[
 0\le f_1(\ux)\le\frac{\beta-1}{\alpha-1}f_2(\ux)
 \et
 f_3(\ux)\le\frac{\alpha^2-\alpha}{\beta-1}f_1(\ux)
\]
for each $\ux\in\{B_1, A_1, A_2, A_3, E_3\}$.  Since $f_1,f_2,f_3$ are
linear forms, these inequalities extend to the set of vertices
$\{\oB_1,\oA_1,\oA_2,\oA_3,E_3\}$ of $\cK$, and thus to their convex
hull $\cK$.  We deduce that, for each $\delta>0$ and each point $\ux$
in $\cK+[-\delta,\delta]^4$, we have
\[
 -c\delta\le f_1(\ux)\le\frac{\beta-1}{\alpha-1}f_2(\ux)+c\delta
 \et
 f_3(\ux)\le\frac{\alpha^2-\alpha}{\beta-1}f_1(\ux)+c\delta,
\]
for a constant $c>0$ that depends only on $\alpha$ and $\beta$.  In
particular, the latter inequalities hold at $\ux=q^{-1}\uP(q)$ for
each $q\ge Q_\delta$, and this yields (i) and (ii).
\end{proof}

\begin{lemma}
\label{min:lemma:intervals}
Let $q,r\in\bR$ with $r>q\ge q_0$.
\begin{itemize}
 \item[(i)] If $P_1$ is constant on $[q,r]$, we have $\kappa_1(t)
    \le (q/t)\kappa_1(q)$ for each $t\in[q,r]$.
 \item[(ii)] If $P_4$ is constant on $[q,r]$, we have $\kappa_3(t)
    \le (q/t)\kappa_3(q)$ for each $t\in[q,r]$.
 \item[(iii)] We have $0\le \kappa_3(t)\le \kappa_2(t)$ for each $t\ge q_0$.
\end{itemize}
\end{lemma}

\begin{proof}
If $P_1$ is constant on $[q,r]$ and if $t\in[q,r]$, we find
\[
 t\kappa_1(t) = \beta P_1(t)-P_2(t) = \beta P_1(q)-P_2(t)
 \le \beta P_1(q)-P_2(q) = q\kappa_1(q),
\]
which proves (i).  The proof of (ii) is similar, and (iii) is clear.
\end{proof}

We conclude with the following lemma which, in view of \eqref{min:eq:equiv},
implies that $\cK(\uP)\subseteq\Deltabar_3$ and thus proves the theorem.

\begin{lemma}
\label{min:lemma:kappa3}
We have $\lim_{q\to\infty} \kappa_3(q)=0$.
\end{lemma}

\begin{proof}
Choose $\epsilon>0$ small enough so that
\begin{equation}
\label{min:lemma:kappa3:eq1}
 1-11\alpha\epsilon>0
 \et
 \frac{\alpha}{\beta}
 \Big(\frac{1+6\alpha\epsilon}{1-11\alpha\epsilon}\Big)<1.
\end{equation}
This is possible since $1<\alpha<\beta$.  Then choose $\delta>0$
such that
\begin{equation}
\label{min:lemma:kappa3:eq2}
 \frac{\alpha}{\beta}
 \Big(\frac{1+6\alpha\epsilon}{1-11\alpha\epsilon}\Big)
 \epsilon
 +
 \Big( \frac{\beta-1}{\alpha-1}+1 \Big)c\delta \le \epsilon,
\end{equation}
and note that $c\delta<\epsilon$.
Since $\oB_1\in\cF(\uP)$ and since each of the linear forms
$f_1$, $f_2$, $f_3$ given by \eqref{min:lemma:kappa:eq1} vanish
at $\oB_1$, there exists $q\ge Q_\delta$ such that
\begin{equation}
\label{min:lemma:kappa3:eq3}
 \kappa_1(q)\le \epsilon
 \et
 \kappa_3(q)\le \kappa_2(q)\le 2\alpha\epsilon.
\end{equation}
For such $q$, let $w$ be the smallest division point of $\uP$
with $w>q$ for which $P_2(w)=P_3(w)$.  We claim that
\begin{equation}
\label{min:lemma:kappa3:eq4}
 \kappa_1(w)\le \epsilon,
 \quad
 \kappa_2(w)\le 2\alpha\epsilon
 \et
 \max_{q\le t\le w}\kappa_3(t)\le 2\alpha\epsilon.
\end{equation}
If we take this for granted, then \eqref{min:lemma:kappa3:eq3}
holds with $q$ replaced by $w$ and, since the division points of
$\uP$ form an infinite discrete sequence in $[q_0,\infty)$, we
conclude, by induction, that $\kappa_3(t)\le 2\alpha\epsilon$
for each $t\ge q_1$ where $q_1$ is the smallest solution of
\eqref{min:lemma:kappa3:eq3} with $q_1\ge Q_\delta$.  The lemma
thus follows from this claim.

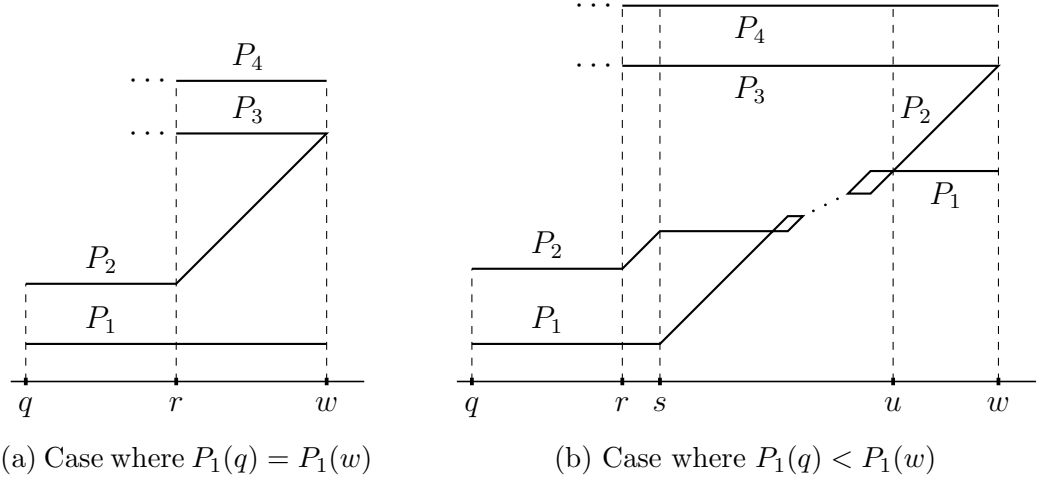
\begin{figure}[h]
\centering
\begin{subfigure}[b]{.3\textwidth}
  \centering
   \begin{tikzpicture}[scale=1]
      \draw[semithick] (-0.2,0) -- (4.5,0);
      \draw[line width=0.05cm] (0,0.05)--(0,-0.05) node[below]{$q$};
      \draw[line width=0.05cm] (2,0.05)--(2,-0.05) node[below]{$r$};
      \draw[line width=0.05cm] (4,0.05)--(4,-0.05) node[below]{$w$};
      \draw[thick] (0,0.5) -- (4,0.5);
      \draw[thick] (0,1.3) -- (2,1.3) -- (4,3.3) -- (2,3.3);
      \draw[thick] (2,4) -- (4,4);
      \draw[dashed] (0,0) -- (0,1.3);
      \draw[dashed] (2,0) -- (2,4);
      \draw[dashed] (4,0) -- (4,4);
      \node[above] at (1,0.5) {$P_1$};
      \node[above] at (1,1.3) {$P_2$};
      \node[above] at (2.95,3.3) {$P_3$};
      \node[above] at (2.95,4) {$P_4$};
      \node at (1.65,3.3) {$\cdots$};
      \node at (1.65,4) {$\cdots$};
  \end{tikzpicture}
  \caption{Case where $P_1(q)=P_1(w)$}
  \label{figPa}
\end{subfigure}%
\begin{subfigure}[b]{.6\textwidth}
   \centering
   \begin{tikzpicture}[scale=1]
      \draw[semithick] (-0.2,0) -- (7.5,0);
      \draw[line width=0.05cm] (0,0.05)--(0,-0.05) node[below]{$q$};
      \draw[line width=0.05cm] (2,0.05)--(2,-0.05) node[below]{$r$};
      \draw[line width=0.05cm] (2.5,0.05)--(2.5,-0.05) node[below]{$s$};
      \draw[line width=0.05cm] (5.6,0.05)--(5.6,-0.05) node[below]{$u$};
      \draw[line width=0.05cm] (7,0.05)--(7,-0.05) node[below]{$w$};
      \draw[thick] (0,0.5) -- (2.5,0.5) -- (4,2) -- (4.2,2) -- (4.4,2.2)
         -- (4.2,2.2) -- (4,2) -- (2.5,2) -- (2,1.5) -- (0,1.5);
      \filldraw (4.54,2.27) circle (0.3pt);
      \filldraw (4.7,2.35) circle (0.3pt);
      \filldraw (4.86,2.43) circle (0.3pt);
      \draw[thick] (7,2.8) -- (5.3,2.8) -- (5,2.5) -- (5.3,2.5)
         -- (7,4.2) -- (2,4.2);
      \draw[thick] (2,5) -- (7,5);
      \draw[dashed] (0,0) -- (0,1.5);
      \draw[dashed] (2,0) -- (2,5);
      \draw[dashed] (2.5,0) -- (2.5,5);
      \draw[dashed] (5.6,0) -- (5.6,5);
      \draw[dashed] (7,0) -- (7,5);
      \node[above] at (1,0.5) {$P_1$};
      \node[above] at (1,1.5) {$P_2$};
      \node[below] at (3.7,4.2) {$P_3$};
      \node[below] at (3.7,5) {$P_4$};
      \node[below] at (6.3,2.8) {$P_1$};
      \node[above] at (5.9,3.3) {$P_2$};
      \node at (1.65,5) {$\cdots$};
      \node at (1.65,4.2) {$\cdots$};
   \end{tikzpicture}
\caption{Case where $P_1(q)<P_1(w)$}
\label{figPb}
\end{subfigure}
\caption{The graph of $\uP$ over $[q,w]$}
\label{figP}
\end{figure}

To prove \eqref{min:lemma:kappa3:eq4}, we first note that the
combined graph of $\uP$ over $[q,w]$ is as in Figure
\ref{figP}.
Indeed, by the choice of $w$, we have $P_2<P_3$ on $(q,w)$,
so $P_1+P_2$ is convex on $[q,w]$: there exists $r\in[q,w]$
such that $P_1+P_2$ is constant on $[q,r]$, and has slope
$1$ on $[r,w]$.  Then $P_1$ and $P_2$ are constant on
$[q,r]$ while $P_3$ and $P_4$ are constant on $[r,w]$. So,
we must have $r<w$.  Let $s$ be the largest element of
$[r,w]$ such that $P_1$ is constant on $[q,s]$.  If $s=w$,
the combined graph of $\uP$ is as in Figure \ref{figPa}.  Otherwise,
it is as in Figure \ref{figPb} where $u$ denotes the largest element
of $[s,w]$ at which $P_1(u)=P_2(u)$.

Since $P_1(q)=P_1(s)$, Lemma \ref{min:lemma:intervals} (i) gives
\begin{equation}
\label{min:lemma:kappa3:eq5}
 \kappa_1(t) \le \frac{q}{t}\kappa_1(q) \le \epsilon
 \quad
 (q\le t\le s).
\end{equation}
By Lemma \ref{min:lemma:kappa} (ii), this in turn implies that
\begin{equation}
\label{min:lemma:kappa3:eq6}
 \kappa_3(t)
 \le \frac{\alpha^2-\alpha}{\beta-1}\kappa_1(t) + c\delta
  < \alpha\epsilon + c\delta
 \le 2\alpha\epsilon
 \quad
 (q\le t\le s),
\end{equation}
since $c\delta\le \epsilon \le \alpha\epsilon$.  If $s=w$,
this proves \eqref{min:lemma:kappa3:eq4} because then $\kappa_2(w)
=\kappa_3(w)=\kappa_3(s)\le 2\alpha\epsilon$.

Suppose from now on that $s<w$.  Since $P_3$ and $P_4$ are constant on
$[r,w]$, we have
\begin{equation}
\label{min:lemma:kappa3:eq7}
 \kappa_3(t)
 = \frac{r}{t}\kappa_3(r)
 \le \kappa_3(r)
 \quad
 (r\le t\le w),
\end{equation}
and so \eqref{min:lemma:kappa3:eq6} yields
$\kappa_3(t)\le 2\alpha\epsilon$ for each $t\in[q,w]$.
In particular, we obtain $\kappa_2(w)=\kappa_3(w)\le
2\alpha\epsilon$ because $P_2(w)=P_3(w)$.  So, it remains only
to prove that $\kappa_1(w)\le \epsilon$.

Applying Lemma \ref{min:lemma:kappa} (i) at the point $w$, and
using $\kappa_2(w)=\kappa_3(w) =
(r/w)\kappa_3(r)$ from above, we obtain
\begin{equation*}
 \kappa_1(w)
 \le \frac{\beta-1}{\alpha-1}\kappa_2(w) + c\delta
  = \frac{\beta-1}{\alpha-1}\cdot\frac{r}{w}\kappa_3(r) + c\delta.
\end{equation*}
Using the first parts of \eqref{min:lemma:kappa3:eq5}
and \eqref{min:lemma:kappa3:eq6} with $t=r$, we also find that
\begin{equation*}
 \kappa_3(r)
 \le \frac{\alpha^2-\alpha}{\beta-1}\kappa_1(r) + c\delta
 \le \frac{\alpha^2-\alpha}{\beta-1}\cdot\frac{q}{r}\kappa_1(q)
     + c\delta.
\end{equation*}
Combining this inequality with the preceding one, we obtain
\begin{equation}
\label{min:lemma:kappa3:eq8}
 \kappa_1(w)
 \le \alpha\frac{q}{w}\kappa_1(q)
     + \Big(\frac{\beta-1}{\alpha-1}+1\Big)c\delta.
\end{equation}
To estimate the ratio $q/w$ from above, we use
Lemma \ref{min:lemma:q} at the
points $q$ and $w$ together with the relations
\begin{equation}
\label{min:lemma:kappa3:eq9}
 P_1(w) = P_2(u) \ge P_2(q) = \beta P_1(q)-q\kappa_1(q)
 \ge \beta P_1(q)-\epsilon q
\end{equation}
coming from the behavior of $\uP$ on $[q,w]$ illustrated in
Figure \ref{figPb}, the definition of $\kappa_1$, and the
hypothesis \eqref{min:lemma:kappa3:eq3}.  Since
$\kappa_2(q)\ge\kappa_3(q)\ge 0$ and
$\kappa_2(w)=\kappa_3(w)\ge 0$, Lemma \ref{min:lemma:q} gives
\[
 q \le (1+3\beta)P_1(q)+(2\kappa_2(q)-3\kappa_1(q))q
 \et
 w \ge (1+3\beta)P_1(w)-3\kappa_1(w)w.
\]
By Lemma \ref{min:lemma:kappa} (i), we have
$\kappa_1(q)\ge -c\delta\ge-\epsilon$, thus
$|\kappa_1(q)|\le \epsilon$ and $2\kappa_2(q)-3\kappa_1(q)
\le 7\alpha\epsilon$ using \eqref{min:lemma:kappa3:eq3}.
By \eqref{min:lemma:kappa3:eq8}, \eqref{min:lemma:kappa3:eq2}
and \eqref{min:lemma:kappa3:eq3},
we also have $\kappa_1(w)\le \alpha\kappa_1(q)+\epsilon
\le 2\alpha\epsilon$.  Together with \eqref{min:lemma:kappa3:eq9},
this gives
\begin{align*}
 (1+6\alpha\epsilon)w
   \ge (1+3\beta)P_1(w)
   &\ge \beta(1+3\beta)P_1(q)-4\beta\epsilon q\\
   &\ge \beta(1-7\alpha\epsilon)q-4\beta\epsilon q
   \ge \beta(1-11\alpha\epsilon)q.
\end{align*}
Substituting in \eqref{min:lemma:kappa3:eq8}, we conclude that
\[
 \kappa_1(w)
 \le
 \frac{\alpha}{\beta}
 \Big(\frac{1+6\alpha\epsilon}{1-11\alpha\epsilon}\Big)
 \kappa_1(q)
 +
 \Big( \frac{\beta-1}{\alpha-1}+1 \Big)c\delta \le \epsilon,
\]
using $\kappa_1(q)\le \epsilon$ and the hypothesis
\eqref{min:lemma:kappa3:eq2}.
\end{proof}

\begin{remark}
Although the above shows that the spectrum $\image^*(\mu_T)$ attached
to a linear map $T\colon\bR^4\to\bR^m$ is not necessarily closed
under the minimum, it is worth looking at conditions on $T$ which
ensures that this property holds.  In his PhD thesis
\cite[Chapter 4]{Ri2019}, the first author shows that it holds
when each component of $T$ achieves its infimum on $\Deltabar$
at the point $E_1$.  An example is the linear map
$T\colon\bR^4\to\bR^3$ given by $T(\ux)=(x_1,x_1+x_2,x_1+x_2+x_3)$.
\end{remark}

%%%%%%%%%%%%%%%%%%%%%%%%%%%%%%%%%%%%
%
%  A family of non-semi-algebraic spectra
%
%%%%%%%%%%%%%%%%%%%%%%%%%%%%%%%%%%%%

\section{A family of non-semi-algebraic spectra}
\label{sec:non-semi-alg}

Let $n\ge 4$ be an integer and let $\alpha>1$ be a real number.
Consider the linear map $T=(T_1,\dots,T_{n+1})\colon\bR^n\to\bR^{n+1}$ whose
components are given by
\begin{align*}
 T_1(\ux)&=x_1,\\
 T_j(\ux)&=\alpha x_j - x_{j+1} \qquad (2\le j\le n-1),\\
 T_n(\ux)&=x_n-\alpha^{n-3}x_2,\\
 T_{n+1}(\ux)&=x_n-\alpha^{n-2}x_1,
\end{align*}
for any $\ux=(x_1,\dots,x_n)\in\bR^n$.  The goal of this paragraph is to show
that its spectrum $\image^*(\mu_T)$ is not a semi-algebraic subset of $\bR^{n+1}$.
More precisely, we will establish the following result where $\bN_+$ denotes
the set of positive integers.

\begin{theorem}
\label{non-semi-alg:thm}
With the above notation, set $\beta=1+\alpha+\cdots+\alpha^{n-2}$, and let
$E$ denote the set of all real numbers $\theta$ for which there exists a proper
$n$-system $\uP\colon[q_0,\infty)\to\bR^n$ with $\mu_T(\uP)=(\theta,0,\dots,0)$.
Then we have
\begin{equation}
 \label{non-semi-alg:eq:E}
 E=\{0\}\cup\{(1+\alpha^m\beta)^{-1}\,;\,m\in\bN_+\}.
\end{equation}
In particular, $E$ contains infinitely many isolated points.  So, $E$
is not a semi-algebraic subset of $\bR$ and thus
$\image^*(\mu_T)$ is not a semi-algebraic subset of $\bR^{n+1}$.
\end{theorem}

As the proof will show, a proper $n$-system $\uP$ with
$\mu_T(\uP)=(\theta,0,\dots,0)$ for some $\theta>0$ has a very
constrained behavior.  Its graph decomposes into pieces which,
after rescaling, converge to a graph of the type shown in Figure
\ref{figH}.  We will need the following lemma.

\begin{lemma}
\label{non-semi-alg:lemma:uniformlimit}
Let $a,b\in\bR$ with $0<a<b$, and let $\uP_k\colon[a,b]\to\bR^n$
be an $n$-system on $[a,b]$ for each $k\in\bN_+$.  Then there
exists a subsequence of $(\uP_k)_{k\ge 1}$ which converges
uniformly on $[a,b]$.  Its limit is a continuous function
$\uf=(f_1,\dots,f_n)\colon [a,b]\to\bR^n$ with the following
properties:
\begin{itemize}
 \item[(i)] for $j=1,\dots,n$, its component $f_j\colon[a,b]\to\bR$
  is $1$-Lipschitz and monotone increasing;%
  \smallskip
 \item[(ii)] for each $t\in[a,b]$, we have $0\le f_1(t)\le \cdots\le f_n(t)$
  and $f_1(t)+\cdots+f_n(t)=t$;
 \smallskip
 \item[(iii)] if $f_j<f_{j+1}$ on $(a,b)$ for some
  $j\in\{1,\dots,n-1\}$, then $f_1+\cdots+f_j$ is convex on $[a,b]$
  and piecewise-linear with slopes $0$ then $1$;
 \smallskip
 \item[(iv)] if we have $f_1(t)<f_2(t)<\cdots<f_n(t)$ for all but
  finitely many $t\in [a,b]$, then $\uf$ is an $n$-system on $[a,b]$.
\end{itemize}
\end{lemma}

\begin{proof}
The sequence $(\uP_k)_{k\ge 1}$ is equicontinuous and uniformly
bounded on $[a,b]$ because it consists of $1$-Lipschitz maps whose
maximum norm is bounded above by $b$.  By Arzel\`a-Ascoli theorem,
its admits therefore a subsequence that converges uniformly on
$[a,b]$.  Let $\uf=(f_1,\dots,f_n)\colon [a,b]\to\bR^n$ be its
limit.  Then $\uf$ is $1$-Lipschitz on $[a,b]$.  In particular,
it is continuous and each of its components $f_1,\dots,f_n$ are
$1$-Lipschitz on $[a,b]$. The latter are also monotonically
increasing on $[a,b]$ since the components of each $\uP_k$ are
such.  This shows Property (i).  Property (ii) is also
immediate because for each $t\in[a,b]$, the coordinates of
$\uP_k(t)$ form a monotone increasing sequence
$P_{k,1}(t)\le \cdots\le P_{k,n}(t)$ with sum $t$.
Now, suppose that $f_j<f_{j+1}$ on $(a,b)$ for some
$j\in\{1,\dots,n-1\}$, and let $[c,d]$ be any compact subinterval
of $(a,b)$.  Then, for each large enough index $k$, we have
$P_{k,j}<P_{k,j+1}$ on $[c,d]$.  For those $k$, the sum
$P_{k,1}+\cdots+P_{k,j}$ is convex on $[c,d]$ with
slopes $0$ then $1$.  We deduce that $f_1+\cdots+f_j$ is convex
on $[c,d]$ and piecewise-linear with slopes $0$ then $1$.
Property (iii) follows from this by letting $c$ and $d$ go to
$a$ and $b$ respectively.  Finally, (iv) follows from (i), (ii)
and (iii).
\end{proof}

\begin{proof}[Proof of Theorem \ref{non-semi-alg:thm}]
It suffices to prove that $E$ is given by
\eqref{non-semi-alg:eq:E}.  We start by proving that the
non-zero points of $E$ have the form $(1+\alpha^m\beta)^{-1}$
for some positive integer $m$.

Let $\uP=(P_1,\dots,P_n) \colon [0,\infty)\to\bR^n$ be a
proper $n$-system such that $\mu_T(\uP)=(\theta,0,\dots,0)$
for some $\theta>0$. We denote by $q_1<q_2<\cdots$
the sequence of points $q\in[1,\infty)$ for which
$P_1(q)=P_2(q)$, listed in increasing order.  In each open
interval $(q_i,q_{i+1})$, there is a point $r_i$ such that
$P_1$ is constant on $[q_i,r_i]$ while it has slope $1$ on
$[r_i,q_{i+1}]$.  As the ratio $P_1(t)/t$ is bounded above
by $1$ for each $t\ge q_0$, it achieves its minimum on
$[q_i,q_{i+1}]$ at the point $r_i$.  By definition of
$\mu_T(\uP)$, this means that
\begin{equation}
 \label{non-semi-alg:thm:proof:eq1}
 \theta
  = \liminf_{q\to\infty} \frac{P_1(q)}{q}
  = \liminf_{i\to\infty} \frac{P_1(r_i)}{r_i}.
\end{equation}
So, for each sufficiently large $i$, we have
$(\theta/2)r_i < P_1(r_i) = P_1(q_i) \le q_i$, and
therefore $1\le r_i/q_i< 2/\theta$.  It follows that
there exists an infinite subset $I$ of $\bN_+$ and a
real number $\rho\in [1,2/\theta]$ such that
$P_1(r_i)/r_i$ and $r_i/q_i$ converge respectively to
$\theta$ and $\rho$ as $i$ goes to infinity in $I$.  Set
\[
 a=1+\beta, \quad c=\rho a \et b=\frac{2}{\theta}a.
\]
For each $i\in I$, we define an $n$-system
$\uP_i\colon[a,b]\to\bR^n$ by
\[
 \uP_i(t)=\frac{a}{q_i}\uP\Big(\frac{q_i t}{a}\Big)
 \quad (a\le t\le b).
\]
By Lemma \ref{non-semi-alg:lemma:uniformlimit}, there is
an infinite subset $I'$ of $I$ such that $P_i$
converges uniformly to a continuous map $\uf=(f_1,\dots,f_n)\colon[a,b]\to\bR^n$
as $i$ goes to infinity in $I'$.  We will show
that $c=1/\theta>a$ and that the restriction
of $\uf$ to $[a,c]$ is an $n$-system
which is uniquely determined by $\theta$.
Then, from the explicit form of $\uf$, we will deduce that
$\theta=(1+\alpha^m\beta)^{-1}$ for some
integer $m\ge 1$.  For the proof, we use freely
the fact that the restriction of $\uf$ to
any closed subinterval of $[a,b]$ satisfies the
properties (i) to (iv) in Lemma
\ref{non-semi-alg:lemma:uniformlimit}.

For each $i\in I$, we note that
\[
 P_{i,1}(a)
   =\frac{a}{q_i}P_1(q_i)
   =\frac{a}{q_i}P_2(q_i)
   =P_{i,2}(a).
\]
Since $P_1$ is $1$-Lipschitz, we also have
\begin{align*}
 P_{i,1}(c)
  = \frac{a}{q_i}P_1(\rho q_i)
  &= \frac{a}{q_i}P_1(r_i+o(q_i))\\
  &= \frac{a}{q_i}P_1(r_i)+o(1)
   = \left\{
      \begin{aligned}
      &\frac{a}{q_i}P_1(q_i)+o(1) = P_{i,1}(a) +o(1),\\
      &a\frac{r_i}{q_i}\frac{P_1(r_i)}{r_i}+o(1)
        = c\theta+o(1).
      \end{aligned}
    \right.
\end{align*}
By passing to the limit as $i$ goes to infinity in $I'$, these
estimates give
\begin{equation}
 \label{non-semi-alg:thm:proof:eq2}
 f_1(a)=f_2(a)
 \et
 f_1(c)=f_1(a)=c\theta.
\end{equation}
As $f_1$ is monotone increasing, the second set of equalities
implies that $f_1$ is constant on $[a,c]$.
Moreover, for fixed $t\in [a,b]$, the ratio $q_it/a$ tends to
infinity with $i$.  So, the hypothesis that $\mu_T(\uP)=(\theta,0,\dots,0)$
yields
\[
 \liminf_{i\to\infty} T\Big(\frac{1}{t}\uP_i(t)\Big)
  = \liminf_{i\to\infty} T\Big(\frac{a}{q_i t}\uP\Big(\frac{q_i t}{a}\Big)\Big)
  \ge (\theta,0,\dots,0),
\]
thus $T(t^{-1}\uf(t))\ge (\theta,0,\dots,0)$.  Explicitly,
this means that $t^{-1}f_1(t)\ge\theta$ and that
\begin{equation}
 \label{non-semi-alg:thm:proof:eq:main}
 \max\{\alpha^{n-2}f_1(t),\alpha^{n-3}f_2(t)\}
 \le f_n(t)\le \alpha f_{n-1}(t)\le \cdots\le \alpha^{n-2}f_2(t).
\end{equation}
Using \eqref{non-semi-alg:thm:proof:eq2}, we deduce that
\begin{equation}
 \label{non-semi-alg:thm:proof:eq3}
 \theta=\frac{f_1(c)}{c}=\min_{a\le t\le b} \frac{f_1(t)}{t},
\end{equation}
and so $f_1(t)>f_1(c)$ when $t>c$.  When $f_1(t)=f_{2}(t)$, the
inequalities \eqref{non-semi-alg:thm:proof:eq:main} force
$\uf(t)$ to be a multiple of $(1,1,\alpha,\dots,\alpha^{n-2})$.
Since the coordinates of this point sum up to $1+\beta=a$ and
since those of $\uf(t)$ sum up to $t$, we deduce that
\begin{equation}
 \label{non-semi-alg:thm:proof:eq4}
 \uf(t)=\frac{t}{a}(1,1,\alpha,\dots,\alpha^{n-2})
 \quad\text{if $f_1(t)=f_2(t)$.}
\end{equation}
In particular, it follows from \eqref{non-semi-alg:thm:proof:eq2}
that
\begin{equation}
 \label{non-semi-alg:thm:proof:eq5}
 \uf(a)=(1,1,\alpha,\dots,\alpha^{n-2})
 \et
 c=\theta^{-1}.
\end{equation}
Similarly, if $f_j(t)=f_{j+1}(t)$ for some integer $j$ with
$2\le j\le n-1$, these inequalities imply that
\begin{equation}
 \label{non-semi-alg:thm:proof:eq6}
 \uf(t)=(r,s,s\alpha,\dots,s\alpha^{j-2},s\alpha^{j-2},\dots,s\alpha^{n-3})
\end{equation}
for some real numbers $r,s$ with $0\le r\le s$, and thus $0\le r<s$
in view of \eqref{non-semi-alg:thm:proof:eq4}.  From this we infer that,
for each $t\in[a,b]$, there is at most one index $j$ with $1\le j< n$
such that $f_j(t)=f_{j+1}(t)$.

The formula \eqref{non-semi-alg:thm:proof:eq4} implies that the
points $t\in[a,b]$ for which $f_1(t)=f_2(t)$ are isolated.  Indeed, if
$t$ is such a point, then we have $f_2<f_3$ in some connected open
neighborhood $U$ of $t$ in $[a,b]$.  Hence, $f_1+f_2$ is convex
on $U$ with slopes $0$ and $1$ (by Lemma
\ref{non-semi-alg:lemma:uniformlimit} (iii)).  This
implies that $f_1(u)+f_2(u)=2u/a$ for at most two values
of $u\in U$ (by comparing slopes since $0<2/a<1$).  By
\eqref{non-semi-alg:thm:proof:eq4}, these are the only possible
$u\in U$ for which $f_1(u)=f_2(u)$.

By the above observation, since $f_1(a)=f_2(a)$, there exists
a maximal $d$ with $a<d\le b$ such that $f_1<f_2$ on $(a,d)$.
Then $f_1$ is convex on $[a,d]$ with slopes
$0$ and $1$.  As $f_1$ is constant on $[a,c]$ and $f_1(t)>f_1(c)$
when $t>c$, we deduce that $c\in [a,d]$ and that $f_1$ has slope $1$
on $[c,d]$.  Since $f_1+\cdots+f_n$ has slope $1$ on $[a,b]$, it
follows that $f_2+\cdots+f_n$ is constant on $[c,d]$ and so each of
$f_2,\dots,f_n$ are constant on $[c,d]$.  In particular, we deduce that
\begin{align*}
 f_2(d) &= f_2(c) \le f_2(a)+c-a = f_1(a)+c-a,\\
 f_1(d) &= f_1(c)+d-c = f_1(a)+d-c.
\end{align*}
Since $f_1(d)\le f_2(d)$, this yields $d\le 2c-a$. We deduce that
$c>a$ because $d>a$, and also that $d<b$ since $2c=2/\theta<b$.
By the choice of $d$, we conclude that $f_1(d)=f_2(d)$, and so $\uf(d)=(d/a)\uf(a)$.

For each $j=1,\dots,n-1$, let $S_j$ denote the closed
subset of $[a,d]$ consisting of all points $t\in [a,d]$
with $f_j(t)=f_{j+1}(t)$.  By an earlier remark,
these sets are pairwise disjoint. Moreover, we have $S_1=\{a,d\}$
and $S_j\cap[c,d]=\emptyset$ for $j=2,\dots,n-1$.  We claim that, like $S_1$,
the sets $S_2,\dots,S_{n-1}$ are also finite.  As the proof will show, this
is where we need $n\ge 4$.

To prove this claim, fix $j\in\{2,\dots,n-1\}$.
For each $t\in S_j$, we have $t\in(a,c)$, thus $f_1(t)=1$, and
so $\uf(t)$ has the form \eqref{non-semi-alg:thm:proof:eq6}
with $r=1$ and $s=(t-1)/\beta_j$ where
\[
 \beta_j=1+\cdots+\alpha^{j-2}+\alpha^{j-2}+\cdots+\alpha^{n-3}.
\]
If $j>2$, there is a neighborhood $U$ of $t$ in $(a,c)$ on which
$f_2<f_3$.  Then $f_1+f_2$ is convex with slopes $0$ and $1$
on $U$, so $f_1(u)+f_2(u)=1+(u-1)/\beta_j$ has at most
two solutions $u$ in $U$ (because $0<1/\beta_j<1$), and therefore
$U\cap S_j$ consists of at most two points.  Similarly, if $j=2$,
there is a neighborhood $U$ of $t$ in $(a,c)$ on which
$f_3<f_4$.  Then $f_1+f_2+f_3$ is convex with slopes $0$ and $1$
on $U$, so $f_1(u)+f_2(u)+f_3(u)=1+2(u-1)/\beta_2$ has at most
two solutions $u$ in $U$ (because $0<2/\beta_2<1$).  Hence
$U\cap S_2$ consists again of at most two points.  In both cases, this
shows that $S_j$ is a discrete subset of $[a,b]$ and so it is
finite.

Since the set $S:=S_1\cup\cdots\cup S_{n-1}$ is finite, Lemma
\ref{non-semi-alg:lemma:uniformlimit} shows that the restriction
of $\uf$ to $[a,d]$ is an $n$-system.  Let $t_0=a<t_1<\cdots<t_N=d$
be the points of $S$ listed in increasing order.  Fix an integer $i$
with $0\le i<N$ and let $k$ denote the index for which
$t_i\in S_k$.  Then the point $\uf(t_i)$ is given by
\eqref{non-semi-alg:thm:proof:eq6} with $r=1$, $j=k$ and
some value of $s$.  We will show that the restriction of $\uf$
to $H=[t_i,t_{i+1}]$ is entirely determined by $\uf(t_i)$.

We first note that, for each $j=1,\dots,n-1$, we have
$f_j<f_{j+1}$ on $(t_i,t_{i+1})$, thus the sum $f_1+\cdots+f_j$
is convex on $H$ with slopes in $\{0,1\}$.

Suppose first that $k\le n-2$.  Since $f_{k+1}$ has slope $1$
immediately to the right of $t_i$, the sum $f_1+\cdots+f_{k+1}$
has constant slope $1$ on $H$.  So, $f_{k+2}$
is constant on $H$ and we have
\[
 \frac{f_{k+2}(t_{i+1})}{f_{k+1}(t_{i+1})}
 < \frac{f_{k+2}(t_i)}{f_{k+1}(t_i)}
 =\alpha.
\]
Thus $t_{i+1}$ belongs to $S_{k+1}$ and $\uf(t_{i+1})$ is given
by \eqref{non-semi-alg:thm:proof:eq6} with $r=1$, $j=k+1$ and the
same value of $s$ as for $\uf(t_i)$.  We conclude that
$\uf_{k+1}$ has slope $1$ on $H$
while all other components of $\uf$ are constant on $H$.  This
situation is illustrated in Figure \ref{figH} (a).

\begin{figure}[h]
   \begin{tikzpicture}[xscale=0.8,yscale=0.8]
      %
      % partie (a)
      \draw[dashed] (0,6) -- (0,0) node[below]{$t_{i}$};
      \draw[dashed] (3,6) -- (3,0) node[below]{$t_{i+1}$};
      \draw[thick] (0,4.5) -- (3,4.5) -- (0,1.5) -- (3,1.5);
       \node[left] at (0,1.5) {$s\alpha^{k-2}$};
       \node[left] at (0,4.5) {$s\alpha^{k-1}$};
       \node[above] at (1.5,4.5) {$f_{k+2}$};
       \node[below] at (1.5,1.5) {$f_k$};
       \node at (1.5,0.3) {$\cdots$};
       \node at (1.5,5.6) {$\cdots$};
       \node[left] at (1.5,3.2) {$f_{k+1}$};
      \node[below] at (1.5,-1) {(a) $k\le n-2$};
      %
      % partie (b)
      \draw[dashed] (6,6) -- (6,0) node[below]{$t_{i}$};
      \draw[dashed] (9,6) -- (9,0) node[below]{$t_{i+1}$};
      \draw[thick] (6,0.5) -- (9,0.5);
       \node[below] at (8.2,0.5) {$f_1$};
       \node[left] at (6,0.5) {$1$};
       \node[right] at (9,0.5) {$1$};
      \draw[thick] (6,1.2) -- (7.5,1.2) -- (9,2.7) -- (6,2.7);
       \node[right] at (7.7,1.4) {$f_2$};
       \node[left] at (6,1.2) {$s$};
       \node[left] at (6,2.7) {$s\alpha$};
       \node[right] at (9,2.7) {$s\alpha$};
       \node[below] at (7,2.7) {$f_3$};
      \draw[thick] (9,4.5) -- (6,4.5) -- (7.5,6) -- (9,6);
       \node[left] at (6,4.5) {$s\alpha^{n-3}$};
       \node[right] at (9,6) {$s\alpha^{n-2}$};
       \node[right] at (9,4.5) {$s\alpha^{n-3}$};
       \node[above] at (8.2,4.5) {$f_{n-1}$};
       \node[left] at (7.05,5.7) {$f_{n}$};
       \node at (6.7,3.6) {$\cdots$};
       \node at (8.3,3.6) {$\cdots$};
      \draw[dashed] (7.5,6) -- (7.5,0) node[below] {$u$};
      \node[below] at (7.5,-1) {(b) $k=n-1$ and $\ell=2$};
      %
      % partie (c)
      \draw[dashed] (13,6) -- (13,0) node[below]{$t_{i}$};
      \draw[dashed] (16,6) -- (16,0) node[below]{$t_{i+1}$};
      \draw[dashed] (14,6) -- (14,0) node[below] {$c$};
      \draw[thick] (13,0.5) -- (14,0.5) -- (16,2.5) -- (13,2.5);
       \node[below] at (15,1.5) {$f_1$};
       \node[below] at (13.5,2.5) {$f_2$};
       \node[left] at (13,0.5) {$1$};
       \node[right] at (16,2.5) {$s$};
       \node[left] at (13,2.5) {$s$};
      \draw[thick] (16,6) -- (14,6) -- (13,5) -- (16,5);
       \node[left] at (13,5) {$s\alpha^{n-3}$};
       \node[right] at (16,5) {$s\alpha^{n-3}$};
       \node[right] at (16,6) {$s\alpha^{n-2}$};
       \node[below] at (15,5) {$f_{n-1}$};
       \node[below] at (15,6) {$f_{n}$};
       \node at (15,3.7) {$\cdots$};
      \node[below] at (14.5,-1) {(c) $k=n-1$ and $\ell=1$};
  \end{tikzpicture}
\caption{All possibilities for the graph of $\uf$ over $[t_i,t_{i+1}]$.}
\label{figH}
\end{figure}
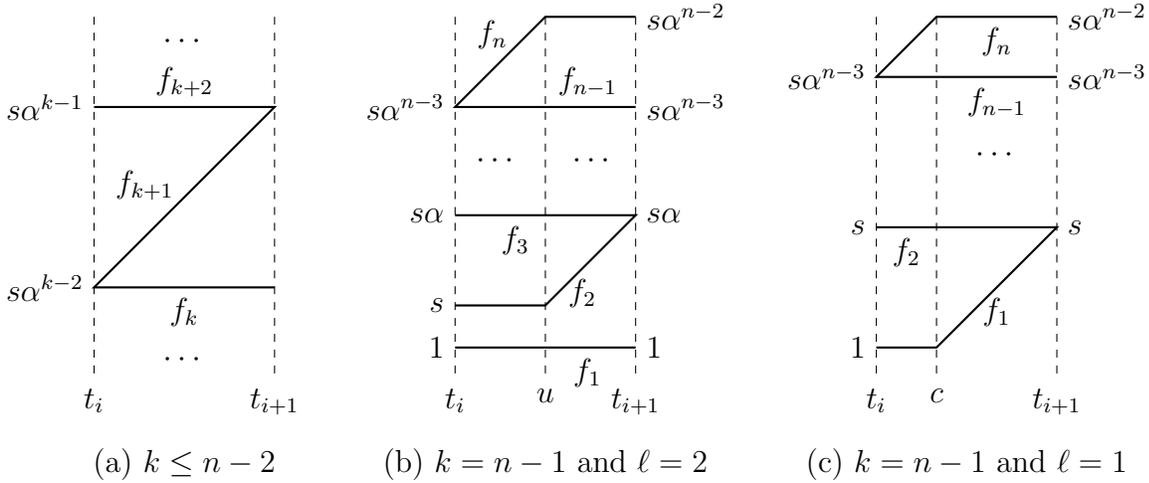

Suppose now that $k=n-1$.  Let $u$ be the largest point of $(t_i,t_{i+1}]$
such that $f_n$ has slope $1$ on $[t_i,u]$.  Since $t_{i+1}\in S$
we must have $u < t_{i+1}$.  Then, $f_1+\cdots+f_{n-1}$
changes slope from $0$ to $1$ at the point $u$ and thus $f_n$
is constant on $[u,t_{i+1}]$.  Let $\ell$ be the smallest index
with $1\le\ell\le n-1$ such that $f_1+\cdots+f_\ell$ changes slope
from $0$ to $1$ at $u$.  Then $f_j$ is constant on $H$
for each $j$ with $\ell<j<n$.  If $\ell=1$, we must have $u=c$
and $t_{i+1}=d\in S_1$.  This is illustrated in Figure
\ref{figH} (c).  Suppose now that $\ell\ge 2$.
Let $v$ be the largest element of $(u,t_{i+1}]$ such
that $f_1+\cdots+f_{\ell-1}$ is constant on $[t_i,v]$.  Then
$f_\ell$ is constant on $[t_i,u]$, has slope $1$ on $[u,v]$
and is constant on $[v,t_{i+1}]$, while $f_{\ell+1}$ is
constant on $[u,t_{i+1}]$.  Applying the main
inequalities \eqref{non-semi-alg:thm:proof:eq:main}
with $t=u$, we deduce that
\[
 \alpha
  \ge \frac{f_{\ell+1}(u)}{f_\ell(u)}
  > \frac{f_{\ell+1}(u)}{f_\ell(v)}
  = \frac{f_{\ell+1}(t_{i+1})}{f_\ell(t_{i+1})},
\]
thus $v=t_{i+1}\in S_{\ell}$.  In particular,
$f_{\ell-1}$ is constant on $[t_i,v]=[t_i,t_{i+1}]$.
If $\ell>2$, this yields
\[
 \alpha
 = \frac{f_\ell(t_i)}{f_{\ell-1}(t_i)}
 < \frac{f_\ell(t_{i+1})}{f_{\ell-1}(t_{i+1})}
\]
which is impossible.  Hence, we must have $\ell=2$ and all components
of $\uf$ other than $f_2$ and $f_n$ are constant on $H=[t_i,t_{i+1}]$.
This is illustrated in Figure \ref{figH} (b).

By the above analysis, the points $\uf(t_i)$ listed
according to their index $i$ are
\begin{equation}
 \label{non-semi-alg:thm:proof:eq7}
 \begin{aligned}
 &\uf(a)=(1,1,\alpha,\dots,\alpha^{n-2}),\\
 &(1,\alpha,\alpha,\dots,\alpha^{n-2}),
  (1,\alpha,\alpha^2,\alpha^2\dots,\alpha^{n-2}),
  \dots,
  (1,\alpha,\dots,\alpha^{n-2},\alpha^{n-2}),\\
 &\qquad \cdots\\
 &(1,\alpha^m,\alpha^m,\dots,\alpha^{m+n-3}),
  (1,\alpha^m,\alpha^{m+1},\alpha^{m+1}\dots,\alpha^{m+n-3}),
  \dots\\
 &\hspace*{200pt}
  \dots,
  (1,\alpha^m,\dots,\alpha^{m+n-3},\alpha^{m+n-3}),\\
 &\uf(d)=(\alpha^m,\alpha^m,\dots,\alpha^{m+n-2}),
 \end{aligned}
\end{equation}
for some integer $m\ge 1$.  The combined graph of $\uf$
on $[a,d]$ (the union of the graphs of its components)
is shown in Figure \ref{fig:pseudoregular} for the case
where $n=4$ and $m=3$.  The switch points of $\uf$ on $(a,d)$ are
\begin{equation}
 \label{non-semi-alg:thm:proof:eq8}
 (1,\alpha,\alpha^2,\dots,\alpha^{n-1}),\,
 (1,\alpha^2,\alpha^3,\dots,\alpha^{n}),\,
 \dots,\,
 (1,\alpha^m,\alpha^{m+1},\dots,\alpha^{m+n-2}).
\end{equation}
In particular, the last switch point is $\uf(c)$, and so we get
\[
 \theta=c^{-1}
   =(1+\alpha^m+\cdots+\alpha^{m+n-2})^{-1}
   =(1+\alpha^m\beta)^{-1}.
\]
This shows that $E\subseteq \{0\} \cup
\{(1+\alpha^m\beta)^{-1}\,;\,m\in\bN_+\}$.

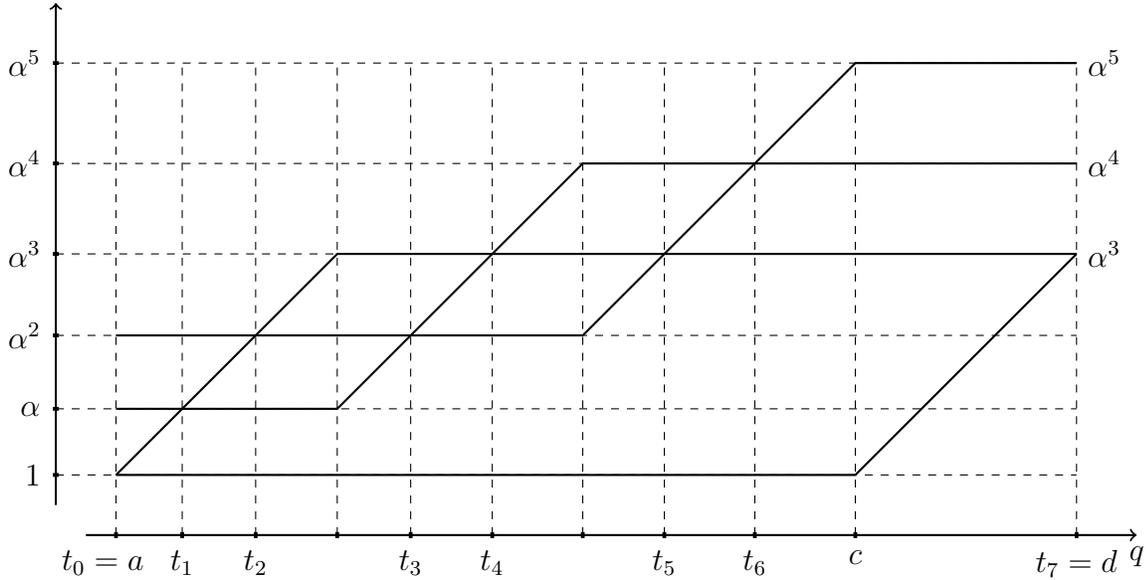
\begin{figure}[h]
   \begin{tikzpicture}[xscale=0.8,yscale=0.8]
      \pgfmathparse{1.11}\let\r\pgfmathresult;
      %
      % axe vertical
      \draw[line width=0.05cm]
        (0.05,10-9)--(-0.05,10-9)
        node[left]{$1$};
      \draw[line width=0.05cm]
        (0.05,10*\r-9)--(-0.05,10*\r-9)
        node[left]{$\alpha$};
      \foreach \q in {2,...,5}
         \draw[line width=0.05cm]
           (0.05,10*\r^\q-9)--(-0.05,10*\r^\q-9)
           node[left]{$\alpha^\q$};
      \draw[->,thick] (0,0.5)--(0,10*\r^5-9+1);
      %
      % axe horizontal
      % premier bloc
      \pgfmathparse{1}
          \let\ti\pgfmathresult;
      \draw[line width=0.05cm]
        (\ti,0.05)--(\ti,-0.05)
        node[below]{$t_0=a$\hspace*{10pt}};
      \draw[line width=0.05cm]
        (\ti+10*\r-10,0.05)--(\ti+10*\r-10,-0.05)
        node[below]{$t_1$};
      \draw[line width=0.05cm]
        (\ti+10*\r^2-10,0.05)--(\ti+10*\r^2-10,-0.05)
        node[below]{$t_2$};
      % second bloc
      \pgfmathparse{\ti+10*(\r^3-1)}
          \let\tii\pgfmathresult;
      \draw[line width=0.05cm]
        (\tii,0.05)--(\tii,-0.05);
      \draw[line width=0.05cm]
        (\tii+10*\r^2-10*\r,0.05)--(\tii+10*\r^2-10*\r,-0.05) node[below]{$t_3$};
      \draw[line width=0.05cm]
        (\tii+10*\r^3-10*\r,0.05)--(\tii+10*\r^3-10*\r,-0.05) node[below]{$t_4$};
      % troisieme bloc
      \pgfmathparse{\tii+10*(\r^4-\r)}
          \let\tiii\pgfmathresult;
      \draw[line width=0.05cm]
        (\tiii,0.05)--(\tiii,-0.05);
      \draw[line width=0.05cm]
        (\tiii+10*\r^3-10*\r^2,0.05)--(\tiii+10*\r^3-10*\r^2,-0.05) node[below]{$t_5$};
      \draw[line width=0.05cm]
        (\tiii+10*\r^4-10*\r^2,0.05)--(\tiii+10*\r^4-10*\r^2,-0.05) node[below]{$t_6$};
      % quatrieme bloc
      \pgfmathparse{\tiii+10*(\r^5-\r^2)}
          \let\tiv\pgfmathresult;
      \draw[line width=0.05cm]
        (\tiv,0.05)--(\tiv,-0.05)
        node[below]{$c$};
      \pgfmathparse{\tiv+10*(\r^3-1)}
          \let\d\pgfmathresult;
      \draw[line width=0.05cm]
        (\d,0.05)--(\d,-0.05) node[below]{$t_7=d$};
      \draw[->,thick] (0.5,0)--(\d+1,0) node[below]{$q$};
      %
      % lignes du graphe conjoint
      \draw[thick] (\ti,1) -- (\tiv,1) -- (\d,10*\r^3-9)
        node[right]{$\alpha^3$};
      \draw[thick] (\ti,1) -- (\tii,10*\r^3-9) -- (\d,10*\r^3-9);
      \draw[thick] (\ti,10*\r-9) -- (\tii,10*\r-9)
        -- (\tiii,10*\r^4-9) -- (\d,10*\r^4-9)
        node[right]{$\alpha^4$};
      \draw[thick] (\ti,10*\r^2-9) -- (\tiii,10*\r^2-9)
        -- (\tiv,10*\r^5-9) -- (\d,10*\r^5-9)
        node[right]{$\alpha^5$};
      % lignes pointillees verticales
      % premier bloc
      \draw[dashed]
        (\ti,0)--(\ti,10*\r^5-9);
      \draw[dashed]
        (\ti+10*\r-10,0)--(\ti+10*\r-10,10*\r^5-9);
      \draw[dashed]
        (\ti+10*\r^2-10,0)--(\ti+10*\r^2-10,10*\r^5-9);
      % second bloc
      \draw[dashed]
        (\tii,0)--(\tii,10*\r^5-9);
      \draw[dashed]
        (\tii+10*\r^2-10*\r,0)--(\tii+10*\r^2-10*\r,10*\r^5-9);
      \draw[dashed]
        (\tii+10*\r^3-10*\r,0)--(\tii+10*\r^3-10*\r,10*\r^5-9);
      % troisieme bloc
      \draw[dashed]
        (\tiii,0)--(\tiii,10*\r^5-9);
      \draw[dashed]
        (\tiii+10*\r^3-10*\r^2,0)--(\tiii+10*\r^3-10*\r^2,10*\r^5-9);
      \draw[dashed]
        (\tiii+10*\r^4-10*\r^2,0)--(\tiii+10*\r^4-10*\r^2,10*\r^5-9);
      % dernier bloc
      \draw[dashed]
        (\tiv,0)--(\tiv,10*\r^5-9);
      \draw[dashed]
        (\d,0)--(\d,10*\r^5-9);
      % lignes pointillees horizontales
      \foreach \y in {0,1,...,5}{
      \draw [dashed] (0,10*\r^\y-9) -- (\d,10*\r^\y-9);}
  \end{tikzpicture}
\caption{The combined graph of $\uf$ on $[a,d]$ for $n=4$ and $m=3$.}
\label{fig:pseudoregular}
\end{figure}

Conversely, for each integer $m\ge 1$, there is
a unique $n$-system $\uf$ on $[a,d]$ with $a=1+\beta$ and
$d=\alpha^ma$, whose division points are given by
\eqref{non-semi-alg:thm:proof:eq7} and
\eqref{non-semi-alg:thm:proof:eq8}.  The first component
$f_1$ of that $n$-system is constant on $[a,c]$ where
$c=1+\alpha^m\beta$, and it has slope $1$ on $[c,d]$.  Therefore
the minimum of $f_1(t)/t$ on $[a,d]$ is $1/c$, achieved at $t=c$.
Moreover, one verifies that $\uf$ satisfies the main conditions
\eqref{non-semi-alg:thm:proof:eq:main} at each $t\in [a,d]$.  More
precisely, we find that
\[
 \min\{t^{-1}T(\uf(t))\,;\, a\le t\le d\} = (c^{-1},0,\dots,0).
\]
Finally, we note that $\uf$ extends uniquely to an $n$-system
on $[a,\infty)$ such that $\uf(\alpha^mt)=\alpha^m\uf(t)$ for each $t\ge a$.
This $n$-system is proper with $\mu_T(\uf)=(c^{-1},0,\dots,0)$.
Thus the set $E$ contains $c^{-1}=(1+\alpha^m\beta)^{-1}$
for each $m\ge 1$.
Since $E$ is a closed subset of $\bR$, it also contains $0$.
This completes the proof of \eqref{non-semi-alg:eq:E} and so
proves the theorem.
\end{proof}

\subsection*{Acknowlegments.}
The research of both authors was partially supported
by the NSERC Discovery grant of the second author.  The first author
was also supported by an Ontario graduate scholarship.

\end{document}